\documentclass[11pt,reqno,tbtags,draft]{amsart}
\usepackage{amssymb}
\usepackage{url}
\usepackage[square,numbers]{natbib}
\bibpunct[, ]{[}{]}{;}{n}{,}{,}

\title[Pareto Record Frontier]
{The Pareto Record Frontier}

\newcommand\urladdrx[1]{{\urladdr{\def~{{\tiny$\sim$}}#1}}}
\author{James Allen Fill}
\address{Department of Applied Mathematics and Statistics,
The Johns Hopkins University,
3400 N.~Charles Street,
Baltimore, MD 21218-2682 USA}
\email{jimfill@jhu.edu}
\urladdrx{http://www.ams.jhu.edu/~fill/}
\thanks{Research for both authors supported by
the Acheson~J.~Duncan Fund for the Advancement of Research in
Statistics.}

\author{Daniel~Q.\ Naiman}
\address{Department of Applied Mathematics and Statistics,
The Johns Hopkins University,
3400 N.~Charles Street,
Baltimore, MD 21218-2682 USA}
\email{daniel.naiman@jhu.edu}
\urladdrx{https://www.ams.jhu.edu/~dan/}

\numberwithin{equation}{section}

\allowdisplaybreaks

\theoremstyle{plain}
\newtheorem{theorem}{Theorem}[section]
\newtheorem{lemma}[theorem]{Lemma}
\newtheorem{proposition}[theorem]{Proposition}
\newtheorem{corollary}[theorem]{Corollary}

\theoremstyle{definition}
\newtheorem{example}[theorem]{Example}
\newtheorem{definition}[theorem]{Definition}

\newtheorem{remark}[theorem]{Remark}

\newtheorem*{acks}{Acknowledgments}

\theoremstyle{remark}

\newenvironment{romenumerate}[1][-10pt]{
\addtolength{\leftmargini}{#1}\begin{enumerate}
 }{\end{enumerate}}

\newcounter{oldenumi}
{\setcounter{oldenumi}{\value{enumi}}
\begin{romenumerate} \setcounter{enumi}{\value{oldenumi}}}
{\end{romenumerate}}

\newcounter{thmenumerate}

\newcounter{xenumerate}   

\newcommand{\refT}[1]{Theorem~\ref{#1}}
\newcommand{\refC}[1]{Corollary~\ref{#1}}
\newcommand{\refL}[1]{Lemma~\ref{#1}}
\newcommand{\refR}[1]{Remark~\ref{#1}}
\newcommand{\refS}[1]{Section~\ref{#1}}

\newcommand{\refP}[1]{Proposition~\ref{#1}}
\newcommand{\refD}[1]{Definition~\ref{#1}}
\newcommand{\refE}[1]{Example~\ref{#1}}
\newcommand{\refF}[1]{Figure~\ref{#1}}

\newcommand\marginal[1]{\marginpar{\raggedright\parindent=0pt\tiny #1}}
\newcommand\REM[1]{{\raggedright\texttt{[#1]}\par\marginal{XXX}}}

\begingroup
  \divide\count255 by 60
  \multiply\count255 by -60
  \advance\count255 by \time
  \ifnum \count255 < 10 \xdef\klockan{\the\count1.0\the\count255}
  \else\xdef\klockan{\the\count1.\the\count255}\fi
\endgroup

\newcommand\nopf{\qed}   

\def\rompar(#1){\textup(#1\textup)}    

\def\xexp(#1){e^{#1}}

\newcommand\half{\tfrac12}

\newcommand\punkt{.\spacefactor=1000}    
\newcommand\iid{i.i.d\punkt}
\newcommand\ie{i.e\punkt}
\newcommand\eg{e.g\punkt}

\newcommand{\as}{a.s\punkt}

\newcommand{\io}{i.o\punkt}

\renewcommand{\aa}{a.a\punkt}

\newcommand{\tend}{\longrightarrow}

\newcommand\Pto{\overset{\mathrm{P}}{\tend}}
\newcommand\Lcto{\overset{\mathcal{L}}{\tend}}
\newcommand\asto{\overset{\mathrm{a.s.}}{\tend}}

\newcommand\bbR{\mathbb R}

\newcommand\bbZ{\mathbb Z}

\newcounter{CC}
\newcounter{cc}

\newcommand\E{\operatorname{\mathbb E{}}}
\renewcommand\P{\operatorname{\mathbb P{}}}
\renewcommand\L{\operatorname{L}}
\newcommand\Var{\operatorname{Var}}

\newcommand\gl{\lambda}

\newcommand\tB{{\widetilde F}}
\newcommand\tC{{\widetilde C}}
\newcommand\tG{{\widetilde G}}
\newcommand\tW{{\widetilde W}}

\newcommand\Bh{B}

\newcommand\doi{D_{01}}

\newcommand\tr{\tilde r}

\newcommand{\ignore}[1]{}

\usepackage{tikz}
\usepackage{amssymb}
\usetikzlibrary{arrows,positioning}
\usepackage[english]{babel}
\usepackage{pgfplotstable}
\usepackage{pgfplots}
\usepackage{adjustbox}
\pgfplotsset{compat=1.3}

\begin{document}

\date{January~25, 2019}

\maketitle

\begin{abstract}
For \iid\ $d$-dimensional observations $X^{(1)}, X^{(2)}, \ldots$ with independent Exponential$(1)$
coordinates, consider the boundary (relative to the closed positive orthant), or ``frontier'', $F_n$ of the closed Pareto record-setting (RS) region
\[
\mbox{RS}_n
:= \{0 \leq x \in {\mathbb R}^d: x \not\prec X^{(i)}\mbox{\ for all $1 \leq i \leq n$}\}
\]
at time $n$, where
$0 \leq x$ means that $0 \leq x_j$ for $1 \leq j \leq d$ and
$x \prec y$ means that $x_j < y_j$ for $1 \leq j \leq d$.  With $x_+ := \sum_{j = 1}^d x_j$, let
\[
F_n^- := \min\{x_+: x \in F_n\} \quad \mbox{and} \quad
F_n^+ := \max\{x_+: x \in F_n\},
\]
and define the \emph{width} of $F_n$ as
\[
W_n := F_n^+ - F_n^-.
\]
We describe typical and almost sure behavior of the processes $F^+$, $F^-$, and $W$.  In particular, we show that  $F^+_n \sim \ln n \sim F^-_n$ almost surely and that $W_n / \ln \ln n$ converges in probability to $d - 1$; and for $d \geq 2$ we show that, almost surely, the set of limit points of the sequence $W_n / \ln \ln n$ is the interval $[d - 1, d]$.

We also obtain modifications of our results that are important in connection with efficient simulation of Pareto records.  Let $T_m$ denote the time that the $m$th record is set.  We show
that $F^+_{T_m} \sim (d! m)^{1/d} \sim F^-_{T_m}$ almost surely and that  
$W_{T_m} / \ln m$ converges in probability to $1 - d^{-1}$;
and for $d \geq 2$ we show that, almost surely, 
the sequence $W_{T_m} / \ln m$ has $\liminf$ equal to $1 - d^{-1}$ and $\limsup$ equal to $1$.
\end{abstract}

\section{Introduction, background, and main results}\label{S:intro}

The study of univariate records is very well developed (\cite{Arnold(1998)} being a classical reference), but that of multivariate records less well so, in part because there are many ways one can formulate the latter concept.  See~\cite{Hwang(2010)}, and the references therein, and~\cite[Chap.~8]{Arnold(1998)} for background.

This paper is mainly about the stochastic process $(F_n)$, where $F_n$ is the boundary, or ``frontier'', for \emph{Pareto records} (otherwise known as \emph{nondominated records} or \emph{weak records}; consult
Definitions~\ref{D:record}--\ref{D:RS}) in general dimension~$d$ when the observed sequence of points $X^{(1)}, X^{(2)}, \dots$ are assumed (as they are throughout the paper)  to be \iid\ (independent and identically distributed) copies of a $d$-dimensional random vector~$X$ with independent Exponential$(1)$ coordinates $X_j$.

Theoretical investigation
leading to the results in this paper were spurred by empirical observations whose generation is discussed briefly in \refS{S:time} (see especially
\refF{fig:simulation}) and in detail in~\cite{Fillgenerating(2018)} and began with the simple result of \refT{T:typical}.

{\bf Notation:\ }Throughout this paper we abbreviate the $k$th iterate of natural logarithm $\ln$ by $\L_k$ and
$\L_1$ by $\L$, and we write
$x_+ := \sum_{j = 1}^d x_j$ for the sum of coordinates of the $d$-dimensional vector $x = (x_1, \dots, x_d)$.

Unless otherwise specifically noted, all the results of this paper hold for any dimension $d \geq 1$.

\subsection{Pareto records and the record-setting region}
\label{S:records}
We begin with some definitions.  Write $x \prec y$
(respectively, $x \leq y$)
to mean that $x_j < y_j$
(resp.,\ $x_j \leq y_j$)
for $1 \leq j \leq d$.
(We caution that, with this convention,  $\leq$ is weaker than $\preceq$, the latter meaning ``$\prec$ or $=$''; indeed, $(0, 0) \leq (0, 1)$ but we have neither $(0, 0) \prec (0, 1)$ nor $(0, 0) = (0, 1)$.  This distinction will matter little in this paper, since the probability that any coordinate of an observation is repeated or vanishes is~$0$, but the distinction is important in~\cite{Fillgenerating(2018)}.)  The notation $x \succ y$ means
$y \prec x$, and $x \geq y$ means $y \leq x$.

\begin{definition}
\label{D:record}
(a)~We say that $X^{(k)}$ is a \emph{(Pareto) record} (or that it \emph{sets} a record at time~$k$) if $X^{(k)} \not\prec X^{(i)}$ for all $1 \leq i < k$.

(b)~If $1 \leq k \leq n$,
we say that $X^{(k)}$ is a \emph{current record} (or \emph{remaining record}, or \emph{maximum}) at time~$n$ if $X^{(k)} \not\prec X^{(i)}$ for all $1 \leq i \leq n$.

(c)~If $1 \leq k \leq n$,
we say that $X^{(k)}$ is a \emph{broken record} at time~$n$ if it is a record but not a current record, that is, if $X^{(k)} \not\prec X^{(i)}$ for all $1 \leq i < k$ but $X^{(k)} \prec X^{(\ell)}$ for some $k < \ell \leq n$; in that case, the observation corresponding to the smallest such~$\ell$ is said to \emph{break} or \emph{kill} the
record $X^{(k)}$.
\end{definition}

For $n \geq 1$ (or $n \geq 0$, with the obvious conventions) let $R_n$ denote the number of records $X^{(k)}$ with $1 \leq k \leq n$, let $r_n$ denote the number of remaining records at time~$n$, and let $\beta_n := R_n - r_n$ denote the number of broken records.  Note that $R_n$ and $\beta_n$ are nondecreasing in~$n$, but the same is not true for $r_n$.  For dimension $d \geq 2$, by
standard consideration of concomitants
[that is, by considering the $d$-dimensional sequence $X^{(1)}, \ldots, X^{(n)}$ sorted from largest to smallest value of (say) last coordinate]
we see that $r_n(d)$ (that is, $r_n$ for dimension~$d$, with similar notation used here for $R_n$) has, for each~$n$, the same (univariate) distribution as $R_n(d-1)$; note, however, the same equality in distribution does \emph{not} hold for the stochastic processes $r(d)$ and $R(d - 1)$.

\begin{definition}
\label{D:RS}
(a)~The \emph{record-setting region} at time~$n$ is the (random) closed set of points
\[
\mbox{RS}_n := \{x \in \bbR^d: 0 \leq x \not\prec X^{(i)}\mbox{\ for all $1 \leq i \leq n$}\}.
\]

(b)~We call the (topological) boundary of $\mbox{RS}_n$ (relative to the closed positive orthant determined by the origin) its \emph{frontier} and denote it by $F_n$.
\end{definition}

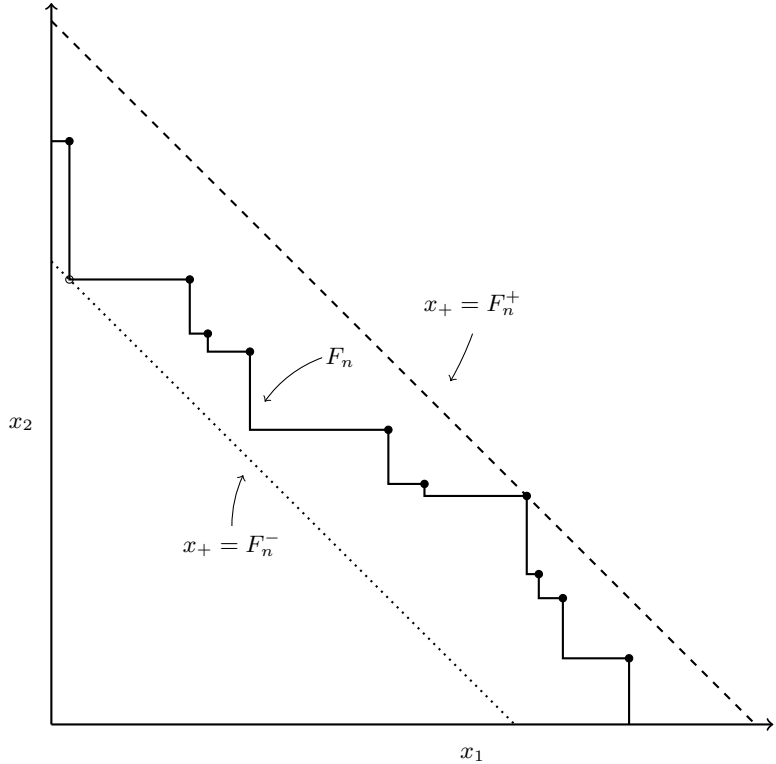
\begin{figure}[htb]
\begin{tikzpicture}[scale=8]
%
%
\draw[->,thick,color=black] (0,0)--(0,1.2);
\draw[->,thick,color=black] (0,0)--(1.2,0);
%
%
\draw[thick,color=black](0,.97)--
(.03,.97)--(.03,.74) --
(.23,.74)--(.23,.65) --
(.26,.65)--(.26,.62) --
(.33,.62)--(.33,.49) --
(.56,.49)--(.56,.40) --
(.62,.40)--(.62,.38) --
(.79,.38)--(.79,.25) --
(.81,.25)--(.81,.21) --
(.85,.21)--(.85,.11) --
(.96,.11)-- (.96,.0)
;
%
%
\filldraw [black]
(.03,.97) circle (.175pt)
(.23,.74) circle (.175pt)
(.26,.65) circle (.175pt)
(.33,.62) circle (.175pt)
(.56,.49) circle (.175pt)
(.62,.40) circle (.175pt)
(.79,.38) circle (.175pt)
(.81,.25) circle (.175pt)
(.85,.21) circle (.175pt)
(.96,.11) circle (.175pt)
;
\draw [black]
(.03,.74) circle (.175pt);
%
%
\draw[dashed,color=black, thick](0,1.17)--(1.17,0);
%
%
\draw[dotted,color=black, thick](0,.77)--(.77,0);
%
%
\draw (.48,.61) node[color=black] {\footnotesize $F_n$};
\draw (.7,.7) node[color=black] {\footnotesize $x_+=F_n^+$};
\draw (.3,.3) node[color=black] {\footnotesize $x_+=F_n^-$};
%
%
\draw [->,color=black](.7,.65) arc (340:330:5mm);
\draw [->,color=black](.3,.33) arc (180:155:2mm);
\draw [->,color=black](.450,.61) arc (110:145:2mm);
%
%
\draw (.7,-.05) node[color=black] {\footnotesize $x_1$};
\draw (-.05,.5) node[color=black] {\footnotesize $x_2$};
\end{tikzpicture}

\caption{Record frontier $F_n$ based on $n$ observations resulting in 10 current records (shown as solid points).
The values $F_n^- = \min\{ x_+ : x \in F_n\}$  and $F_n^+ = \max\{ x_+ : x \in F_n\}$ determine two hyperplanes
$x_+ = F_n^-$ and $x_+ = F_n^+.$  A new observation sets a record if and only if it falls in the region to the upper right of $F_n.$
}
\label{fig:frontierdefs}
\end{figure}

\begin{remark}
\label{R:RS}
The terminology in \refD{D:RS}(a) is natural since the next observation $X^{(n + 1)}$ sets a record if and only if it falls in the record-setting region.  Note that
\begin{align*}
\mbox{RS}_n
&= \{x \in \bbR^d: 0 \leq x \not\prec X^{(i)}\mbox{\ for all $1 \leq i \leq n$} \\
&{} \qquad \qquad \qquad \mbox{such that $X^{(i)}$ is a current record at time~$n$}\},
\end{align*}
and that the current records at time~$n$ all belong to~$\mbox{RS}_n$ but lie on its frontier.
Observe
also that $F_n$ is a closed subset of $\mbox{RS}_n$.  Because this paper makes heavy use of the classical probabilistic notion of boundary-crossing probabilities, to avoid confusion we have chosen to use the term ``frontier'' for $F_n$, rather than ``boundary'', in \refD{D:RS}(b).
\end{remark}

\subsection{The record-setting frontier}
\label{S:boundary}

Our first result shows that deviations of the sum of coordinates for a generic current record at time~$n$ from
$\L n$ are typically of constant order.
Observe that the conditional distribution of $X^{(k)}_+$ given that $X^{(k)}$ is a current record at time~$n$ doesn't depend on $k \in \{1, \dots, n\}$; in particular, it's the conditional distribution of $X^{(n)}_+$ given that $X^{(n)}$ sets a record.  Let $Y_n$ be a random variable with that distribution.
Let~$G$ denote a random variable with the standard Gumbel distribution (\ie, distribution function $x \mapsto e^{- e^{-x}}$, $x \in \bbR$), and write $\Lcto$ for convergence in law (\ie, in distribution)

\begin{theorem}
\label{T:typical}
We have
\[
Y_n - \L n \Lcto G.
\]
\end{theorem}

\begin{proof}
This is quite elementary.  Let $p_n$ denote the probability that $X^{(n)}$ sets a record.  Fix $n \geq 2$ for the moment.  For $x \succ 0$ we have
\begin{align*}
\lefteqn{\P(X^{(n)} \in dx\,|\,X^{(n)} \not\prec X^{(i)}\mbox{\ for all $1 \leq i \leq n$})} \\
&= p_n^{-1} \P(X^{(n)} \in dx,\,X^{(n)} \not\prec X^{(i)}\mbox{\ for all $1 \leq i \leq n$}) \\
&= p_n^{-1} \P(X^{(n)} \in dx,\,x \not\prec X^{(i)}\mbox{\ for all $1 \leq i \leq n - 1$}) \\
&= p_n^{-1} \P(X^{(n)} \in dx) \P(x \not\prec X^{(i)}\mbox{\ for all $1 \leq i \leq n - 1$}) \\
&= p_n^{-1} e^{-x_+} [1 - \P(x \prec X^{(1)})]^{n - 1}\,dx
= p_n^{-1} e^{-x_+} (1 - e^{-x_+})^{n - 1}\,dx,
\end{align*}
and so the conditional density depends on~$x$ only through~$x_+$.  It follows that the density $f_n(y)$ of $Y_n$ satisfies
\[
f_n(y) = p_n^{-1} \frac{y^{d - 1}}{(d - 1)!} e^{-y} (1 - e^{-y})^{n - 1}, \quad y > 0.
\]
Using the well-known asymptotic equivalence $p_n \sim n^{-1} (\L n)^{d - 1} / (d - 1)!$ as $n \to \infty$ [see~\eqref{sets} below], it is easy to check that, for each fixed $z \in \bbR$, the density of $Y_n - \L n$ at~$z$ converges to the standard Gumbel density $e^{- z} e^{- e^{- z}}$ as $n \to \infty$.  The claimed result thus follows from Scheff\'{e}'s theorem (\eg,\ \cite[Thm.~16.12]{Billingsley(2012)}), which shows that there is in fact convergence in total variation.
\end{proof}

This paper primarily concerns the stochastic process $(F_n)$, and specifically its ``width'' as defined next
(see \refF{fig:frontierdefs}).

\begin{definition}
\label{D:W}
Recall that $F_n$ denotes the frontier of $\mbox{RS}_n$, and let
\begin{equation}
\label{-+}
F_n^- := \min\{x_+: x \in F_n\} \quad \mbox{and} \quad F_n^+ := \max\{x_+: x \in F_n\}.
\end{equation}
We define the \emph{width} of $F_n$ as
\begin{equation}
\label{W}
W_n := F_n^+ - F_n^-.
\end{equation}
\end{definition}

Very roughly put, what we will see in this paper is that, unlike $Y_n$ of \refT{T:typical}, deviations of $F^+_n$ from $\L n$ are exactly of order $\L_2 n$; on the other hand, we will see that deviations
of $F^-_n$ from $\L n$ are of smaller order than $\L_2 n$.  It will follow that the width of the frontier is exactly of order $\L_2 n$.

We next make some simple observations about the quantities appearing in \refD{D:W} that will prove fundamentally useful to our development.

\begin{lemma}[\underline{characterization of $F_n^+$}]
\label{L:+}
We have
\[
F_n^+ = \max\{X^{(k)}_+: 1 \leq k \leq n\},
\]
which is nondecreasing in~$n$.
\end{lemma}

\begin{proof}
The current records at time~$n$ all belong to $F_n$, and broken records and non-records all have coordinate-sums (strictly) smaller than some current record.  Thus
$F_n^+ \geq \max\{X^{(k)}_+: 1 \leq k \leq n\}$.  Conversely, if $x \in F_n$, then $x \preceq X^{(i)}$ for some~$i$; it follows that $F_n^+ \leq \max\{X^{(k)}_+: 1 \leq k \leq n\}$.
\end{proof}

\begin{lemma}[\underline{two upper bounds on $F_n^-$}]
\label{L:-}
\

{\rm (a)}~Define
\[
B_n^+(j) := \max\{X_j^{(i)}: 1 \leq i \leq n\}.
\]
Then
\[
F_n^- \leq \min_{1 \leq j \leq d} B_n^+(j).
\]

{\rm (b)}~Let $1 \leq m \leq n$.  Define
\[
\Bh_{m, n} := \mbox{\rm $m^{\rm th}$-largest value among $X^{(k)}_+$ with $1 \leq k \leq n$}.
\]
Then, over the event $\{r_n \geq m\}$ that there are at least~$m$ remaining records at time~$n$, we have
\[
F_n^- \leq \Bh_{m, n}.
\]

{\rm (c)}~The
processes $F^-$, $\min_{1 \leq j \leq d} B^+(j)$, and $\Bh_{m, \cdot}$ (for any~$m$) all have nondecreasing sample paths.
\end{lemma}

\begin{proof}
(a)~For $j = 1, \dots, d$, let $i_j \in \{1, \dots, n\}$ denote the almost surely unique index such that
\[
X^{(i_j)}_j = \max\{X_j^{(i)}: 1 \leq i \leq n\}.
\]
Let $e_j = (0, \dots, 0, 1, 0, \dots, 0)$ denote the $j$th coordinate vector.
We claim that the points $Y^{(j)} := X^{(i_j)} e_j$ with $j = 1, \dots, d$ all belong to $F_n$ (in fact, to
$\mbox{$F_n \cap \mbox{RS}_n$}$), and then the inequality is immediate.  To prove the claim, note that all of
the points $Y^{(j)}$ belong to $\mbox{RS}_n$ [because $Y^{(j)}_j = X^{(i_j)}_j$ and hence $Y^{(j)} \not\prec X^{(i_j)}$] but also to $F_n$ [because $Y^{(j)} \leq X^{(i_j)}$].

(b)~Over the event $\{r_n \geq m\}$, $F_n^-$ is certainly at most the $m$th-largest sum of coordinates of remaining records, which is in turn at most $\Bh_{m, n}$.

(c)~The asserted monotonicity is clear for the bounding processes.  The asserted monotonicity
of $F^-$ follows easily from the observation that
$F_{n + 1} \subseteq \mbox{RS}_{n + 1} \subseteq \mbox{RS}_n$.
\end{proof}

It seems difficult to study the processes
$F^+$ and $F^-$ bivariately, so we draw all our conclusions about the width process~$W$ by studying
$F^+$ and $F^-$ univariately (that is, separately) and
using $W = F^+ - F^-$.
The behavior of
$F^+$ is well known from classical extreme value theory and is reviewed in \refS{S:+}.  Conclusions about
$F^-$ will be drawn from (i)~the upper-bounding processes in~\refL{L:-}(a)--(b) together with
classical extreme value theory for those bounding processes and (ii)~a rather nontrivial lower bound developed in \refS{S:-}.

\subsection{Main results}
\label{S:main}

We next present the main results of our paper.  What the results show, in various precise senses, is that $F_n^+$ and $F_n^-$ both concentrate near $\L n$, with deviations  that are $O(\L_2 n)$, from which it follows of course that $W_n = O(\L_2 n)$.  But for $d \geq 2$ we show more, namely, that $\L_2 n$ is the exact scale for $W_n$, that is, that $W_n = \Theta(\L_2 n)$.
We can even narrow things down further:\
$W_n / \L_2 n \to d - 1$ in probability for each $d \geq 1$, with an almost sure $\liminf$ equal to $d - 1$ and an almost sure $\limsup$ equal to~$d$.

Here are our main results for arbitrary but fixed dimension $d \geq 1$.  We consider both convergence in probability (typical behavior) and almost sure largest and smallest deviations from $\L n$ (top and bottom boundary-behavior, respectively) for large~$n$.

\begin{theorem}[Kiefer~\cite{Kiefer(1972)}]
\label{T:+}
Consider the process $F^+$ defined at~\eqref{-+}.
\medskip

{\rm (a)~\underline{Typical behavior of $F^+$}:}
\[
F_n^+ - [\L n + (d - 1) \L_2 n - \L((d - 1)!)] \Lcto G.
\]
\smallskip

{\rm (b)~\underline{Top boundaries for $F^+$}:}
\[
\P(F_n^+ \geq \L n + c \L_2 n\mbox{\rm \ \io}) =
\begin{cases}
1 & \mbox{{\rm if} $c \leq d$;} \\
0 & \mbox{{\rm if} $c > d$}.
\end{cases}
\]
\smallskip

{\rm (c)~\underline{Bottom boundaries for $F^+$}:}
\[
\P(F_n^+ \leq \L n + (d - 1) \L_2 n - \L_3 n - \L((d-1)!) + c\mbox{\rm \ \io}) =
\begin{cases}
1 & \mbox{{\rm if} $c \geq 0$;} \\
0 & \mbox{{\rm if} $c < 0$}.
\end{cases}
\]
\end{theorem}
\medskip

\refT{T:+} gives rise immediately to the following succinct corollary.
\medskip

\begin{corollary}[Kiefer~\cite{Kiefer(1972)}]
\label{C:+}
Consider the process $F^+$ defined at~\eqref{-+}.
\medskip

{\rm (a)~\underline{Typical behavior of $F^+$}:}
\[
\frac{F_n^+ - \L n}{\L_2 n} \Pto d - 1.
\]

{\rm (b)~\underline{Almost sure behavior for $F^+$}:}
\[
\liminf \frac{F_n^+ - \L n}{\L_2 n} = d - 1 < d =
\limsup \frac{F_n^+ - \L n}{\L_2 n} \mbox{\rm \ \as}
\]
\end{corollary}
\medskip

\begin{remark}
\label{R:interval+}
In fact, one can show rather simply from \refC{C:+}(b) and the fact that
$F^+$ has nondecreasing sample paths that the set (call it $\Lambda$) of limit points of the
sequence $(F_n^+ - \L n) / \L_2 n$ is almost surely the closed interval
$[d - 1, d]$.
Here
is a sketch of the proof.  The set~$\Lambda$ is closed, so we need only show that~$\Lambda$ is dense in
$[d - 1, d]$, which clearly follows if we can show that
\begin{equation}
\label{one step}
\limsup_{n \to \infty}
\left[ \frac{F^+_n - \L n}{\L_2 n} - \frac{F^+_{n + 1} - \L (n + 1)}{\L_2 (n + 1)} \right] \leq 0\mbox{\ \as},
\end{equation}
the roughly stated idea being that then (\as) the
sequence $(F_n^+ - \L n) / \L_2 n$ ``can't leap downward over any interval \io" in its infinitely many downward moves from its $\limsup$ to its $\liminf$.  To
prove~\eqref{one step}, we first
bound $F^+_{n + 1}$ from below by $F^+_n$, then express the resulting difference with a common denominator, and finally use the consequence $F^+_n \sim \L n\mbox{\ \as}$ of \refC{C:+}(b) to find
\begin{align*}
\lefteqn{\hspace{-0.1in}\frac{F^+_n - \L n}{\L_2 n} - \frac{F^+_{n + 1} - \L (n + 1)}{\L_2 (n + 1)}} \\
&\leq \frac{(1 + o(1)) (n \L n)^{-1} F^+_n + (1 + o(1)) n^{-1} \L_2 n}{(1 + o(1)) (\L_2 n)^2}
\sim n^{-1} (\L_2 n)^{-1} = o(1) \mbox{\ \as}
\end{align*}
as $n \to \infty$.
\end{remark}
\medskip

\begin{remark}
\label{R:Bai2005}
Our \refT{T:+} formalizes and improves
upon related computations in Bai et al.~\cite[Secs.~1 and 3.2]{Bai(2005)} who, for the limited purpose of proving a central limit theorem reviewed in \refT{T:typical counts}(a) below, ``observe that nearly all maxima occur in a thin strip sandwiched between [the] two parallel hyper-planes''
\[
x_+ = \L n - \L_3 n - \L[4 (d - 1)] \quad \mbox{and} \quad x_+ = \L n + 4 (d - 1) \L_2 n.
\]
\end{remark}
\medskip

Our results
for $F^-$ show that the deviations of $F^-_n$ from $\L n$ are almost surely negligible on a scale of $\L_2 n$.
\medskip

\begin{theorem}
\label{T:-}
Consider the process $F^-$ defined at~\eqref{-+}.
\medskip

{\rm (a)~\underline{Typical behavior of $F^-$}:}
\[
\P(F_n^- \leq \L n - 3 \L_3 n) \to 0
\]
and
\[
\P(F_n^- \geq \L n + c_n) \to 0\mbox{\rm \ if $c_n \to \infty$}.
\]
\medskip

{\rm (b)~\underline{Top outer boundaries for $F^-$}:}  If $d \geq 2$, then
\[
\P(F_n^- \geq \L n + c \L_2 n\mbox{\rm \ \io}) = 0 \mbox{\rm \ if $c > 0$}.
\]
\medskip

{\rm (c1)~\underline{A bottom outer boundary for $F^-$ on the scale of $\L_3 n$}:}
\[
\P(F_n^- \leq \L n - 3 \L_3 n\mbox{\rm \ \io}) = 0.
\]
\medskip

{\rm (c2)~\underline{A bottom inner boundary for $F^-$ on the scale of $\L_3 n$}:}
\[
\P(F_n^- \leq \L n - \L_3 n\mbox{\rm \ \io}) = 1.
\]
\end{theorem}
\medskip

\refT{T:-} gives rise immediately to the following succinct corollary.
\medskip

\begin{corollary}
\label{C:-}
Consider the process $F^-$ defined at~\eqref{-+}.
\medskip

{\rm (a)~\underline{Typical behavior of $F^-$}:}
\[
\frac{F_n^- - \L n}{\L_2 n} \Pto 0.
\]

{\rm (b)~\underline{Almost sure behavior for $F^-$}:\ }If $d \geq 2$, then
\[
\lim \frac{F_n^- - \L n}{\L_2 n} = 0\mbox{\rm \ \as}
\]
\end{corollary}
\medskip

We come now to our main focus, the process~$W$.  The results in \refT{T:W} follow directly from Corollaries~\ref{C:+} and~\ref{C:-}.
\medskip

\begin{theorem}
\label{T:W}
Consider the process $W$ defined at~\eqref{W}.
\medskip

{\rm (a)~\underline{Typical behavior of $W$}:}
\[
\frac{W_n}{\L_2 n} \Pto d - 1.
\]

{\rm (b)~\underline{Almost sure behavior for~$W$}:\ }If $d \geq 2$, then
\[
\liminf \frac{W_n}{\L_2 n} = d - 1 < d =
\limsup \frac{W_n}{\L_2 n} \mbox{\rm \ \as},
\]
and, in particular,
\[
W_n = \Theta(\L_2 n) \mbox{\rm \ \as}
\]
\end{theorem}
\medskip

\begin{remark}
(a)~When $d = 1$, at each time $n \geq 1$ there is exactly one current record, $F_n^+ = F_n^-$ is
the value of that record, $\mbox{RS}_n$ is the closed
interval $[F_n^+, \infty)$, and $W_n = 0$.

(b)~Using \refR{R:interval+}, \refT{T:W}(b) can be strengthened to the conclusion that the set of limit points of the sequence $W_n / \L_2 n$ is almost surely the closed interval $[d - 1, d]$.

(c)~\refT{T:W}(b) has the following immediate corollary.  If, for some positive integer $d_0$, processes $W(d)$ corresponding to dimension~$d$, $d = d_0, d_0 + 1, \dots$, are defined on a common probability space (regardless of any dependence among the processes), then
\begin{equation}
\label{double limit}
\lim_{d \to \infty} \limsup_{n \to \infty} \frac{W_n(d)}{(d - 1) \L_2 n} = 1 = \lim_{d \to \infty} \liminf_{n \to \infty} \frac{W_n(d)}{(d - 1) \L_2 n} \mbox{\rm \ \ \as}
\end{equation}
That is, roughly speaking, for time~$n$ large relative to large dimension~$d$, the width $W_n(d)$ almost surely concentrates near $(d - 1) \L n$.

(d)~We could have used~$d$ in the denominators of~\eqref{double limit}, but we chose $d - 1$ because of \refT{T:W}(a).  A remark of a somewhat similar flavor as~(b) for convergence in probability is the following.  If, for some integer $d_0 \geq 2$, processes $W(d)$ corresponding to dimension~$d$, $d = 2, \ldots, d_0$, are defined on a common probability space (regardless of any dependence among the processes), then
\[
\max_{2 \leq d \leq d_0} \left| \frac{W_n(d)}{(d - 1) \L_2 n} - 1 \right| \Pto 0.
\]
We have not investigated whether this result might extend to dimension~$d_0$ growing with~$n$.
\end{remark}

\subsection{Outline of paper}
\label{S:outline}
The
stochastic process $F^+$ is studied in \refS{S:+}, where we prove \refT{T:+}.  We treat the
process $F^-$ in \refS{S:-}, where we prove \refT{T:-}.  In \refS{S:counts} we assess asymptotic behavior of the record counts $R_n$, $r_n$, and $\beta_n$ introduced following \refD{D:record} as preparation for \refS{S:time}, where we produce versions of our main results concerning the record-setting frontier
process~$F$ when time is measured in the number of records (rather than observations $X^{(i)}$) generated.

\section{The process $F^+$}
\label{S:+}
This
section is devoted to the proof of \refT{T:+} concerning the process $F^+$ defined at \eqref{-+}.  In light of the characterization provided by \refL{L:+}, \refT{T:+} follows from results of~\cite{Kiefer(1972)}.  Kiefer is concerned with behavior of the law of the iterated logarithm type for the empirical distribution function and sample $p_n$-quantiles for a sequence of independent uniform$(0, 1)$ random variables, with $p_n > 0$ and $p_n \downarrow 0$, but notes that his results ``may easily be translated into results for general laws.''  Since we are concerned here with a sequence $X^{(1)}_+, X^{(2)}_+, \dots$ from the Gamma$(d, 1)$ distribution and with (only) the $p_n = 1 / n$ \emph{upper} quantile, for completeness and the reader's convenience we distill Kiefer's proof(s) for our special case.

\begin{proof}[Proof of \refT{T:+}]
(a)~This is elementary.  We have
\begin{align*}
\lefteqn{\hspace{-.5in}\P(F^+_n - [\L n + (d - 1) \L_2 n - \L((d - 1)!)] \leq x)} \\
&= \left[ \P\left( X^{(1)}_+ - [\L n + (d - 1) \L_2 n - \L((d - 1)!)] \leq x \right) \right]^n \\
&= \left[ \P\left(X^{(1)}_+ \leq \L n + (d - 1) \L_2 n  - \L((d - 1)!) + x \right) \right]^n \\
&= \left( 1 - \sum_{j = 0}^{d - 1} e^{- \gl} \frac{\gl^j}{j!} \right)^n
= \left[ 1 - (1 + o(1)) e^{- \gl} \frac{\gl^{d - 1}}{(d - 1)!} \right]^n \\
&= \left[ 1 - (1 + o(1)) n^{-1} e^{- x} \right]^n \to e^{-e^{-x}} = \P(G \leq x),
\end{align*}
where $\gl := \L n + (d - 1) \L_2 n - \L((d - 1)!) + x$.
\ignore{
This sentence gave further details in the first proof in~(b) below.
{\bf Suppress the next sentence in final draft:}
Note that
\begin{align*}
\P(X^{(n)}_+ > b_n)
&= \P(X^{(1)}_+ > b_n) = \sum_{j = 0}^{d - 1} e^{- b_n} \frac{b_n^j}{j!} \\
&\sim e^{- b_n} \frac{b_n^{d - 1}}{(d - 1)!}
\sim [(d - 1)!]^{-1} n^{-1} (\L n)^{d - 1 -c}.
\end{align*}
}

(b)~Kiefer describes two proofs.  The first proof observes, for
any sequence $b_n \to \infty$ which
is ultimately monotone nondecreasing, that
\[
\{F^+_n > b_n\mbox{\ \io}\} = \{X^{(n)}_+ > b_n\mbox{\ \io}\}
\]
and applies the Borel--Cantelli lemmas to the sequence of independent events $\{X^{(n)}_+ > b_n\}$ with
$b_n \equiv \L n + c \L_2 n$.
The second proof exploits the nondecreasingness of the sample paths of the process
$F^+_{\cdot} = \Bh_{1, \cdot}$ noted in \refL{L:-} and proceeds as follows.  If $(b_n)$ is ultimately monotone nondecreasing and $(n_j)$ is any strictly increasing sequence of positive
integers, then
\[
\{F_{n_j, n_{j + 1}}^+ \geq b_{n_{j + 1}}\mbox{\rm\ \io$(j)$}\}
 \subseteq \{F_n^+ \geq b_n\mbox{\rm\ \io$(n)$}\}
 \subseteq \{F_{n_{j + 1}}^+ \geq b_{n_j}\mbox{\rm\ \io$(j)$}\},
\]
where we note that the random variables
\begin{equation}
\label{B2}
F_{n_j, n_{j + 1}}^+ \equiv \max\{X^{(k)}_+: n_j < k \leq n_{j + 1}\}
\end{equation}
are independent.  Now choose $b_n \equiv \L n + c \L_2 n$ and $n_j \equiv 2^j$ and apply the Borel--Cantelli lemmas.
\ignore{
{\bf Suppress the next two sentences in final draft:}
Note that if $c > d$, then
\begin{align*}
\P(F_{n_{j + 1}}^+ \geq b_{n_j})
&= 1 - \P(F_{n_{j + 1}}^+ < b_{n_j}) \\
&= 1 - \left[ \P\left( X^{(1)}_+ < b_{n_j} \right) \right]^{n_{j + 1}} \\
&= 1 - \left( 1 - \sum_{k = 0}^{d - 1} e^{-b_{n_j}} \frac{b_{n_j}^k}{k!} \right)^{n_{j + 1}} \\
&= 1 - \left[ 1 - (1 + o(1)) e^{-b_{n_j}} \frac{b_{n_j}^{d - 1}}{(d - 1)!} \right]^{n_{j + 1}} \\
&= 1 - \left[ 1 - (1 + o(1)) [(d - 1)!]^{-1} n_j^{-1} (\L n_j)^{d - 1 - c} \right]^{n_{j + 1}} \\
&= 1 - \exp \left[ - (1 + o(1)) \frac{n_{j + 1} (\L n_j)^{d - 1 - c}}{(d - 1)! n_j} \right] \\
&= 1 - \exp \left[ - (1 + o(1)) \frac{2 (\L 2)^{d - 1 - c}}{(d - 1)!} j^{d - 1 - c} \right] \\
&\sim \frac{2 (\L 2)^{d - 1 - c}}{(d - 1)!} j^{d - 1 - c}
\end{align*}
is summable; and that if $d - 1 < c \leq d$, then the series with $j$th term equal~to
\begin{align*}
\lefteqn{\hspace{-.3in}\P(F_{n_j, n_{j + 1}}^+ \geq b_{n_{j + 1}})} \\
&= 1 - \P(F_{n_{j + 1} - n_j}^+ < b_{n_{j + 1}}) \\
&= 1 - \left[ \P\left( X^{(1)}_+ < b_{n_{j + 1}} \right) \right]^{n_{j + 1} - n_j} \\
&= 1 - \left( 1 - \sum_{k = 0}^{d - 1} e^{-b_{n_{j + 1}}} \frac{b_{n_{j + 1}}^k}{k!} \right)^{n_{j + 1} - n_j} \\
&= 1 - \left[ 1 - (1 + o(1)) e^{-b_{n_{j + 1}}} \frac{b_{n_{j + 1}}^{d - 1}}{(d - 1)!} \right]^{n_{j + 1} - n_j} \\
&= 1 - \left[ 1 - (1 + o(1)) [(d - 1)!]^{-1} n_{j + 1}^{-1} (\L n_{j + 1})^{d - 1 - c} \right]^{n_{j + 1} - n_j} \\
&= 1 - \exp \left[ - (1 + o(1)) \frac{(n_{j + 1} - n_j) (\L n_{j + 1})^{d - 1 - c}}{(d - 1)! n_{j + 1}} \right] \\
&= 1 - \exp \left[ - (1 + o(1)) \frac{(\L 2)^{d - 1 - c}}{2 (d - 1)!} (j + 1)^{d - 1 - c} \right] \\
&\sim \frac{(\L 2)^{d - 1 - c}}{2 (d - 1)!} j^{d - 1 - c}
\end{align*}
diverges.  [The assumption $c > d - 1$ is without loss of generality for proving the case $c \leq d$ of
\refT{T:+}(b).]
}

(c)~For the case $c < 0$ of outer-class bottom boundaries, we start with the observation that if $(b_n)$ is ultimately monotone nondecreasing and $(n_j)$ is any strictly increasing sequence of positive integers, then
\[
\{F_n^+ \leq b_n\mbox{\rm\ \io$(n)$}\}
 \subseteq \{F_{n_j}^+ \leq b_{n_{j + 1}}\mbox{\rm\ \io$(j)$}\}.
\]
We then choose $b_n \equiv \L n + (d - 1) \L_2 n - \L_3 n - \L((d-1)!) + c$ with $c < 0$ and
$n_j \equiv \lfloor e^{|c| j / 2} \rfloor$ and apply the
first Borel--Cantelli lemma.
\ignore{
{\bf Suppress the next sentence in final draft:}
Indeed, note that if $c < 0$, then
\begin{align*}
\P(F_{n_j}^+ \leq b_{n_{j + 1}})
&= \left[ \P\left( X^{(1)}_+ \leq b_{n_{j + 1}} \right) \right]^{n_j} \\
&= \left( 1 - \sum_{k = 0}^{d - 1} e^{-b_{n_{j + 1}}} \frac{b_{n_{j + 1}}^k}{k!} \right)^{n_j} \\
&= \left[ 1 - (1 + o(1)) e^{-b_{n_{j + 1}}} \frac{b_{n_{j + 1}}^{d - 1}}{(d - 1)!} \right]^{n_j} \\
&= \left[ 1 - (1 + o(1)) e^{|c|} n_{j + 1}^{-1} \L_2 n_{j + 1} \right]^{n_j} \\
&= \exp \left[ - (1 + o(1)) e^{|c|}\,\frac{n_j \L_2 n_{j + 1}}{n_{j + 1}} \right] \\
&= \exp \left[ - (1 + o(1)) e^{-|c| / 2} e^{|c|} \L j \right] = j^{- (1 + o(1)) e^{|c| / 2}}
\end{align*}
is summable.
}

For the case $c \geq 0$ of inner-class bottom boundaries, we start with the observation that if $(b_n)$ is ultimately monotone nondecreasing and $(n_j)$ is any strictly increasing sequence of positive integers, then, recalling the definition~\eqref{B2},
\begin{align*}
\lefteqn{\hspace{-.6in}\{F_{n_j}^+ \leq b_{n_{j + 1}}\mbox{\rm\ \aa$(j)$}\} \cap
\{F_{n_j, n_{j + 1}}^+ \leq b_{n_{j + 1}}\mbox{\rm\ \io$(j)$}\}} \\
&\subseteq \{F_{n_{j + 1}}^+ \leq b_{n_{j + 1}}\mbox{\rm\ \io$(j)$}\}
\subseteq \{F_n^+ \leq b_n\mbox{\rm\ \io$(n)$}\}.
\end{align*}
We then choose $b_n \equiv \L n + (d - 1) \L_2 n - \L_3 n - \L((d-1)!) + c$ with $c \geq 0$ and
$n_j \equiv \lfloor e^{\alpha j \L j} \rfloor$ with $\alpha > 1$ and apply the first Borel--Cantelli lemma to the events $\{F_{n_j}^+ > b_{n_{j + 1}}\}$ and the second Borel--Cantelli lemma to the independent events
$\{F_{n_j, n_{j + 1}}^+ \leq b_{n_{j + 1}}\}$.
\ignore{
{\bf Suppress the next sentence in final draft:}
Indeed, note that if $c \geq 0$, then
\begin{align*}
\P(F_{n_j}^+ > b_{n_{j + 1}})
&= 1 - \P(F_{n_j}^+ \leq b_{n_{j + 1}}) \\
&= 1 - \left[ \P\left( X^{(1)}_+ \leq b_{n_{j + 1}} \right) \right]^{n_j} \\
&= 1 - \left( 1 - \sum_{k = 0}^{d - 1} e^{-b_{n_{j + 1}}} \frac{b_{n_{j + 1}}^k}{k!} \right)^{n_j} \\
&= 1 - \left[ 1 - (1 + o(1)) e^{-b_{n_{j + 1}}} \frac{b_{n_{j + 1}}^{d - 1}}{(d - 1)!} \right]^{n_j} \\
&= 1 - \left[ 1 - (1 + o(1)) e^{-c} n_{j + 1}^{-1} \L_2 n_{j + 1} \right]^{n_j} \\
&= 1 - \exp \left[ - (1 + o(1)) e^{-c}\,\frac{n_j \L_2 n_{j + 1}}{n_{j + 1}} \right] \\
&\sim e^{-c}\,\frac{n_j \L_2 n_{j + 1}}{n_{j + 1}}
\sim e^{- (c + \alpha)} j^{- \alpha} \L j
\end{align*}
is summable, and when $c = 0$ [which we may assume without loss of generality for proving the case $c \geq 0$ of \refT{T:+}(c)] the series with $j$th term equal~to
\begin{align*}
\P(F_{n_j, n_{j + 1}}^+ \leq b_{n_{j + 1}})
&= \P(F_{n_{j + 1} - n_j}^+ \leq b_{n_{j + 1}}) \\
&= \left[ \P\left( X^{(1)}_+ \leq b_{n_{j + 1}} \right) \right]^{n_{j + 1} - n_j} \\
&= \left( 1 - \sum_{k = 0}^{d - 1} e^{-b_{n_{j + 1}}} \frac{b_{n_{j + 1}}^k}{k!} \right)^{n_{j + 1} - n_j} \\
&= \left[ 1 - [1 + O((j \L j)^{-1})] e^{-b_{n_{j + 1}}} \frac{b_{n_{j + 1}}^{d - 1}}{(d - 1)!} \right]^{n_{j + 1} - n_j} \\
&= \left[ 1 - [1 + O((j \L j)^{-1})] n_{j + 1}^{-1} \L_2 n_{j + 1} \right]^{n_{j + 1} - n_j} \\
&\sim [\alpha (j + 1) \L(j + 1)]^{-1}
\end{align*}
diverges, where the asymptotic equivalence holds because as $j \to \infty$ we have
\begin{align*}
\lefteqn{\hspace{-.1in}-\L \P(F_{n_j, n_{j + 1}}^+ \leq b_{n_{j + 1}})} \\
&= (n_{j + 1} - n_j) \times - \L \left[ 1 - [1 + O((j \L j)^{-1})] n_{j + 1}^{-1} \L_2 n_{j + 1} \right] \\
&= (n_{j + 1} - n_j) \left[ [1 + O((j \L j)^{-1})] n_{j + 1}^{-1} \L_2 n_{j + 1}
+ O(n_{j + 1}^{-2} (\L_2 n^{j + 1})^2) \right] \\
&= (n_{j + 1} - n_j) n_{j + 1}^{-1} \L_2 n_{j + 1} + o(1) \\
&= \L (j + 1) + \L_2 (j + 1) + \L \alpha + o(1).
\end{align*}
}
\end{proof}

\section{The process $F^-$}
\label{S:-}

\subsection{Towards a stochastic lower bound on $F_n^-$}
\label{S:TLB-}

To prove \refT{T:-} we need a stochastic lower bound on $F_n^-$ to complement the upper bound of \refL{L:-}.  For this we use the definitions of the
frontier $F_n$ and the closed record-setting region $\mbox{RS}_n$ to argue as follows.  For $x \in \bbR^d$, let
\[
O^+_x := \{y \in \bbR^d: y \succ x\}
\]
denote the open positive orthant determined by~$x$.  For any set $S \subseteq \bbR^d$, let $N_n(S)$ denote the number of observations $X^{(i)}$ with $1 \leq i \leq n$ that
fall in~$S$.  Then
\begin{align}
\{F_n^- \leq b\}
&= \{x_+ \leq b\mbox{\ for some $x \in F_n$}\}
= \{x_+ \leq b\mbox{\ for some $x \in \mbox{RS}_n$}\} \nonumber \\
&= \{x_+ \leq b\mbox{\ for some $x \geq 0$ satisfying $x \not\prec X^{(i)}$ for all $1 \leq i \leq n$}\} \nonumber \\
&= \bigcup_{x \geq 0:\,x_+ \leq b} \{N_n(O^+_x) = 0\} \nonumber \\
\label{union}
&= \bigcup_{x \geq 0:\,x_+ = b} \{N_n(O^+_x) = 0\}.
\end{align}
The difficulty with upper-bounding the probability of this event is of course that the last union is uncountable.  In the next subsection we produce a geometric lemma whose application effectively bounds the uncountable union by a finite union.

\begin{figure}[htb]
\begin{tikzpicture}[scale=7]
%
%
\draw[->,thick,color=black] (0,0)--(0,1);
\draw[->,thick,color=black] (0,0)--(1,0);
%
%
\draw (.5,-.05) node[color=black] {\footnotesize $x_1$};
\draw (-.05,.5) node[color=black] {\footnotesize $x_2$};
%
%
\draw[color=black, dashed, thick](.9,0)--(0,.9);
%
%
\draw[color=black, dashdotted, thick](.8,0)--(0,.8);
%
%
\filldraw [black]
(.9,.0) circle (.175pt)
(.8,.1) circle (.175pt)
(.7,.2) circle (.175pt)
(.6,.3) circle (.175pt)
(.5,.4) circle (.175pt)
(.4,.5) circle (.175pt)
(.3,.6) circle (.175pt)
(.2,.7) circle (.175pt)
(.1,.8) circle (.175pt)
(.0,.9) circle (.175pt)
%
%
(.671,.129) circle (.175pt)
;
%
%
\draw[color=black, thick, dotted] (.671,.129)--(1,.129);
\draw[color=black, thick, dotted] (.671,.129)--(.671,1.);
\draw(.80,.25) node[color=black] {\normalsize $O_v^+$};
\draw(.64,.11) node[color=black] {\small $v$};


\draw (.40,.78) node[color=black] {\scriptsize $x_+=2m-(d-1)$};
\draw (.22,.37) node[color=black] {\scriptsize $x_+=2m-2(d-1)$};
%
%
%
%
\draw [->,color=black](.210,.78) arc (90:100:4mm);
%
%
\draw [->,color=black](.210,.4) arc (120:110:6mm);

\end{tikzpicture}

\caption{
Geometric lemma illustrated for $d = 2$.  Given $v$ with $v_+ = 2 m - 2(d - 1)$, the orthant $O^+_v$ determined by~$v$ must contain a point~$i$ with integer coordinates on the hyperplane $x^+ = 2 m - (d - 1)$.}
\label{fig:geomlemma}
\end{figure}
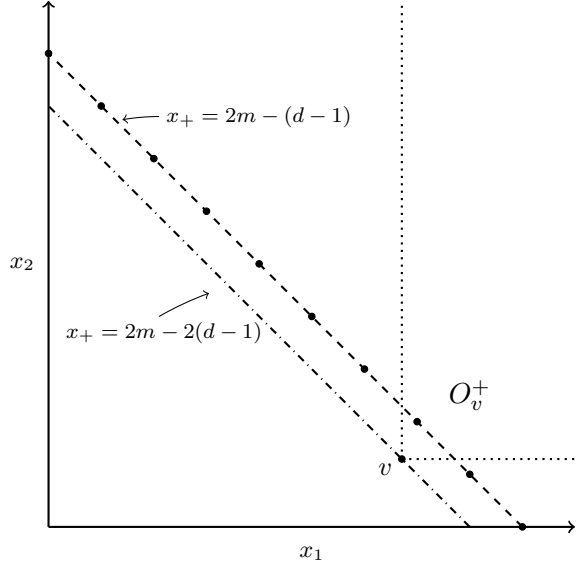

\subsection{A geometric lemma}
\label{S:geo}

Consider the (uncountable) union of positive orthants whose vertices lie on the
hyperplane $x_+ = 2 m - 2 (d - 1)$ in $\bbR^d,$ where $m \geq d - 1$ is an integer.
We can also form a finite union of positive orthants whose vertices lie on the hyperplane $x+ = 2m - (d-1)$
situated a bit further from the origin.
Our key geometric lemma guarantees that the uncountable union contains the finite union (see \refF{fig:geomlemma}).

\begin{lemma}
\label{L:geo}
Given a positive integer $m \geq d - 1$, and $0 \leq x \in \bbR^d$ with
\begin{equation}
\label{x+}
x_+ = 2 m - 2 (d - 1),
\end{equation}
there exists $0 \leq i \in \bbZ^d$ with
\begin{equation}
\label{i+}
i_+ = 2 m - (d - 1)
\end{equation}
such that
\begin{equation}
\label{orth}
O^+_i \subseteq O^+_x.
\end{equation}
\end{lemma}

\begin{proof}
We need to prove the existence of $0 \leq i \in \bbZ^d$ satisfying~\eqref{i+} and~\eqref{orth} (i.e.,  $x \leq i$).  The frugal choice $0 \leq i' \in \bbZ^d$ defined by
\[
i'_j := \left\lceil x_j \right\rceil, \quad j = 1, \dots, d,
\]
satisfies~\eqref{orth} but not necessarily~\eqref{i+}.  However, using~\eqref{x+} we observe that $i'_+$ is at least the integer
\[
x_+ = 2 m - 2 (d - 1)
\]
and strictly less than the integer $2 m - 2 (d - 1) + d = 2 m - (d - 2)$, \ie, is at most $2 m - (d - 1)$.  Thus we need only (arbitrarily) ``sweeten'' (\ie, add~$1$ to) precisely $2 m - (d - 1) - i'_+ \in \bbZ \cap [0, d - 1]$ of the entries $i'_j$ to obtain~$i$ with the desired properties.
\end{proof}

\subsection{A stochastic lower bound on $F_n^-$}
\label{S:LB-}

Let $0 \leq b < \L n$.
Returning to~\eqref{union}, we now see from \refL{L:geo} with $t = \L n \geq 0$ and
\[
m = \left\lceil \frac{(d - 1) \L n}{\L n - b} \right\rceil \geq d - 1,
\]
together with homogeneity
[$O^+_{c y} = c\,O^+_y$ for $0 \leq y \in \bbR^d$ and $0 \leq c \in \bbR^1$], that
\begin{align*}
\{F_n^- \leq b\}
&= \bigcup_{x \geq 0:\,x_+ = b} \{N_n(O^+_x) = 0\} \\
&\subseteq \bigcup_{x \geq 0:\,x_+ = \left( 1 - \frac{d - 1}{m} \right) t} \{N_n(O^+_x) = 0\} \\
&\subseteq \bigcup_{0 \leq i \in \bbZ^d:\,i_+ = 2 m - (d - 1)}
\left\{ N_n\left( O^+_{\frac{t}{2 m} i} \right) = 0 \right\},
\end{align*}
and so by finite subadditivity
\[
\P(F_n^- \leq b)
\leq \sum_{0 \leq i \in \bbZ^d:\,i_+ = 2 m - (d - 1)} \P\left( N_n \left( O^+_{\frac{t}{2 m} i} \right) = 0 \right).
\]
But
\begin{align*}
\P\left( N_n \left( O^+_{\frac{t}{2 m} i} \right) = 0 \right)
&= \P\left( X \notin O^+_{\frac{t}{2 m} i} \right)^n
= \left[ 1 - \P\left( X \in O^+_{\frac{t}{2 m} i} \right) \right]^n \\
&= \left[1 - \exp\left( - \tfrac{t}{2 m} i_+ \right) \right]^n \\
&= \left[ 1 - \exp\left\{ - \left( 1 - \tfrac{d - 1}{2 m} \right) \L n \right\} \right]^n \\
&= \left[1 - n^{- \left( 1 - \frac{d - 1}{2 m} \right)} \right]^n \\
&\leq \exp\!\left(- n^{\frac{d - 1}{2 m}} \right).
\end{align*}
Since the cardinality of $\{0 \leq i \in \bbZ^d:\,i_+ = 2 m - (d - 1)\}$ equals
\[
{2 m \choose d - 1} \leq \frac{(2 m)^{d - 1}}{(d - 1)!}
\]
we conclude that
\begin{align*}
\P(F_n^- \leq b)
&\leq \frac{(2 m)^{d - 1}}{(d - 1)!} \exp\!\left(- n^{\frac{d - 1}{2 m}} \right) \\
&\leq (1 + o(1)) \frac{[2 (d - 1)]^{d - 1}}{(d - 1)!}
\left( 1 - \frac{b}{\L n} \right)^{- (d - 1)} \\
&{} \qquad \qquad \times \exp\!\left[- \exp\left\{ (1 + o(1)) \half (\L n - b) \right\} \right],
\end{align*}
where the last inequality holds assuming that $b = b_n = (1 + o(1)) \L n$ as $n \to \infty$.

We summarize and simplify the bound we have derived in the next proposition, where we assume further that $\L n - b_n \to \infty$.  The bound is the key to the proof of the first assertion in \refT{T:-}(a) and
of \refT{T:-}(c1).
\begin{proposition}[\underline{Stochastic lower bound on $F_n^-$}]
\label{P:-lower}
Let $0 \leq b_n < \L n$ with $b_n = (1 - o(1)) \L n$ and $\L n - b_n \to \infty$.
Then
\[
\P(F_n^- \leq b_n)
\leq
(\L n)^{d - 1} \exp\!\left[- \exp\left\{ (1 + o(1)) \half (\L n - b_n) \right\} \right],
\]
\nopf
\end{proposition}

\subsection{Proof of \refT{T:-}}
\label{S:proof-}
In this subsection we prove \refT{T:-}, part by part in the order (a), (c1), (c2), (b).

\begin{proof}[Proof of \refT{T:-}{\rm (a)}]
The second assertion in \refT{T:-}(a) follows from the case $d = 1$ of \refT{T:+}(a) since, according to
\refL{L:-}(a), we have
\begin{equation}
\label{Bbound}
F_n^- \leq \min_{1 \leq j \leq d} B_n^+(j) \leq B_n^+(1),
\end{equation}
where we recall the definition
\[
B^+_n(j) := \max\{X^{(i)}_j: 1 \leq i \leq n\}.
\]
The first assertion follows from part~(c1), proved next.
\end{proof}

\begin{proof}[Proof of \refT{T:-}{\rm (c1)}]
As noted in \refL{L:-}, the process $F^-$ has nondecreasing sample paths.  From this it follows that if $(b_n)$ is (ultimately) monotone nondecreasing and $(n_j)$ is any strictly increasing sequence of positive integers, then
\[
\{F_n^- \leq b_n\mbox{\rm\ \io$(n)$}\} \subseteq \{F_{n_j}^- \leq b_{n_{j + 1}}\mbox{\rm\ \io$(j)$}\}.
\]
To complete the proof, we choose $b_n \equiv \L n - 3 \L_3 n$ and $n_j \equiv 2^j$, bound
$\P(F_{n_j}^- \leq b_{n_{j + 1}})$ using \refP{P:-lower}, and apply the first Borel--Cantelli lemma.

Here are the details.  Since $\L n_j = j \L 2$ and
\[
b_{n_{j + 1}} = (j + 1) \L 2 - 3 \L_2[(j + 1) \L 2] = j \L 2 - (1 + o(1)) 3 \L_2 j,
\]
the hypotheses of \refP{P:-lower} are met and
\begin{align*}
\P(F_{n_j}^- \leq b_{n_{j + 1}})
&\leq (j \L 2)^{d - 1} \exp\!\left[- \exp\left\{ (1 + o(1)) \tfrac{3}{2} \L_2 j \right\} \right] \\
&= \exp\!\left[- (\L j)^{(1 + o(1))(3 / 2)} \right],
\end{align*}
which is summable.
\end{proof}

\begin{remark}
\label{R:3}
We chose the constant~$3$ as the coefficient of $- \L_3 n$ in parts~(a) and~(c1) of \refT{T:-} for convenience.  As the proof shows, we could have used any constant larger than~$2$.
\end{remark}

\begin{proof}[Proof of \refT{T:-}{\rm (c2)}]
This follows immediately from the case $d = 1$ of \refT{T:+}(c) using the aforementioned bound~\eqref{Bbound}.
\end{proof}

There remains only the proof of \refT{T:-}(b).  For that we need first the following almost sure lower bound on $r_n$, which is of interest in its own right.

\begin{theorem}
\label{T:r}
Assume $d \geq 2$.  Let $r_n$ denote the number of remaining records at time~$n$.  Then
\[
\liminf \frac{r_n}{(\L n) / (d \L_2 n)} \geq 1\mbox{\rm \ \as}
\]
\end{theorem}

\begin{proof}
Fix $\epsilon > 0$.
From \refC{C:+}(b) with $d = 1$ it follows that almost surely
\[
\liminf \frac{B^+_n(1) - \L n}{\L_2 n} = 0
\]
and hence $B^+_n(1) \geq \L n - \epsilon \L_2 n$ a.a.\ \
Additionally, from the now-established \refC{C:+}(b) and \refT{T:-}(c1), it follows that almost surely
\[
\limsup \frac{W_n}{\L_2 n} \leq d
\]
and hence $W_n / \L_2 n \leq (1 + \epsilon) d$ a.a.\ \

Label the remaining records in (\as\ strictly) increasing order of first coordinate as $Z^{(1)}, \dots, Z^{(r_n)}$, and define $Z^{(0)} := Y^{(2)}$ as defined in the proof of \refL{L:-}(a).  Note in particular that the points
$Z^{(i)}$ with $0 \leq i \leq r_n$ all belong to $F_n$, that $Z^{(0)}_1 = Y^{(2)}_1 = 0$, and that
$Z^{(r_n)}_1 = B^+_n(1)$.  Therefore,
\begin{align*}
\L n - \epsilon \L_2 n
&\leq B^+_n(1)
= Z^{(r_n)}_1 - Z^{(0)}_1
= \sum_{i = 1}^{r_n} \left( Z^{(i)}_1 - Z^{(i - 1)}_1 \right) \\
&\leq r_n W_n
\leq (1 + \epsilon)\,d\,r_n \L_2 n
\end{align*}
for all large~$n$, almost surely.  The desired result follows.
\end{proof}

\begin{proof}[Proof of \refT{T:-}{\rm (b)}]
In light of \refT{T:r} and \refL{L:-}(b), it is sufficient that for each fixed positive integer~$m$ we have
\begin{equation}
\label{Bh}
\P\left( \Bh_{m, n} \geq \L n + \frac{a}{m} \L_2 n\mbox{\rm \ \io} \right) = 0
\end{equation}
if $a > 1$.  But~\eqref{Bh} is known from \cite[Thm.~1, see esp.~(3.1)]{Kiefer(1972)}.
\end{proof}

\section{Record counts}
\label{S:counts}

Knowledge about the record counts $R_n$, $r_n$, and $\beta_n$ discussed in \refS{S:intro} is interesting in its own right, and knowledge about $R_n$ will be needed in the next section.

\subsection{Typical behavior}
\label{S:typical}
In this subsection we review a known central limit theorem (CLT) of Berry--Esseen type for $r_n$ and use it to derive easily CLTs for $R_n$ and $\beta_n$.  Here are the results.  Complicated but explicit forms are known for the constants $\gamma_{d, j}$ appearing in the variance expressions.

\begin{theorem}[Bai et al.~\cite{Bai(2005), Bai(1998)}]
\label{T:typical counts}
Let~$\Phi$ denote the standard normal distribution function.
\smallskip

{\rm (a)} Let $d \geq 2$.  Then there exist constants $\gamma_{d, j}$ with $\gamma_{d, 0} \geq 1 / (d - 1)! > 0$ such that the number $r_n$ of remaining records at time~$n$ satisfies
\begin{align*}
\E r_n &= (\L n)^{d - 1} \sum_{j = 0}^{d - 1} \frac{(-1)^j \Gamma^{(j)}(1)}{j! (d - 1 - j)!} (\L n)^{-j}
+ O(n^{-1} (\L n)^{d - 1}) \sim \frac{(\L n)^{d -1}}{(d -1)!}, \\
\Var r_n &= (\L n)^{d - 1} \sum_{j = 0}^{d - 1} \gamma_{d, j} (\L n)^{-j}
+ O(n^{-1} (\L n)^{2 d - 2}) \sim \gamma_{d, 0} (\L n)^{d - 1},
\end{align*}
and
\[
\sup_x \left| \P\left( \frac{r_n - \E r_n}{\sqrt{\Var r_n}} < x \right) - \Phi(x) \right|
= O((\L n)^{- (d - 1) / 4} (\L_2 n)^d).
\]
\smallskip

{\rm (b)} Let $d \geq 1$.  Then the number $R_n$ of records set through time~$n$ satisfies
\begin{align*}
\E R_n &= (\L n)^d \sum_{j = 0}^d \frac{(-1)^j \Gamma^{(j)}(1)}{j! (d - j)!} (\L n)^{-j}
+ O(n^{-1} (\L n)^d) \sim \frac{(\L n)^d}{d!}, \\
\Var R_n &= (\L n)^d \sum_{j = 0}^d \gamma_{d + 1, j} (\L n)^{-j}
+ O(n^{-1} (\L n)^{2 d}) \sim \gamma_{d + 1, 0} (\L n)^d,
\end{align*}
and
\[
\sup_x \left| \P\left( \frac{R_n - \E R_n}{\sqrt{\Var R_n}} < x \right) - \Phi(x) \right|
= O((\L n)^{-  d / 4} (\L_2 n)^{d + 1}).
\]
\smallskip

{\rm (c)} Let $d \geq 1$.  Then the number $\beta_n = R_n - r_n$ of broken records at time~$n$ satisfies
\begin{align*}
\E \beta_n
&= (\L n)^d \left[ \frac{1}{d!}
+ \sum_{j = 1}^d \frac{(-1)^j [\Gamma^{(j)}(1) + j \Gamma^{(j - 1)}(1)]}{j! (d - j)!} (\L n)^{-j} \right] \\
&{} \qquad + O(n^{-1} (\L n)^d) \sim \frac{(\L n)^d}{d!}, \\
\Var \beta_n &= \gamma_{d + 1, 0} (\L n)^d [1 + O((\L n)^{-1/2})],
\end{align*}
and the central limit theorem
\[
\frac{\beta_n - \E \beta_n}{\sqrt{\Var \beta_n}}\mbox{\rm \ converges in law to standard normal}.
\]
\end{theorem}

\begin{proof}
Part~(a) is known from \cite{Bai(2005)}:\ their eq.~(8) for $\E r_n$, their Theorem~1 for $\Var r_n$, their eq.~(13)---and the main theorem of~\cite{Bai(1998)}---for the stated lower bound on
$\gamma_{d, 0}$, and their Theorem~2 for the CLT.
\smallskip

Part~(b) follows immediately from part~(a) by use of concomitants.  (Recall the discussion concerning concomitants preceding \refD{D:RS}.)
\smallskip

For $d = 1$, part~(c) follows from part~(b) because $r_n = 1$ for $n \geq 1$.  For $d \geq 2$, part~(c) follows from parts~(a) and~(b); for $\Var \beta_n$ we use the triangle inequality for $L^2$-norm after centering by means, and for the CLT we use the CLT of part~(b) together with Slutsky's theorem.
\end{proof}

We have not attempted to find further terms in the asymptotic expansion for $\Var \beta_n$ nor a
Berry--Esseen theorem for $\beta_n$.

\subsection{Almost sure behavior}
\label{S:as}

We next establish a sufficient condition for a top boundary for the absolute centered process
$(|R_n - \E R_n|)$ to be of outer class, and derive from that condition strong-law concentration for~$R$ about its mean function.  We also establish analogous results for the processes~$\beta$ and~$r$.

\begin{theorem}
\label{T:as}
Let $d \geq 1$.

{\rm (a)}~If $\epsilon > 0$, then
\[
\P\left( |R_n - \E R_n| \geq (\L n)^{\frac{3d}{4} + \epsilon}\mbox{\rm \ \io} \right) = 0.
\]
As a consequence,
\[
\frac{R_n}{\E R_n} \asto 1.
\]

{\rm (b)}~If $\epsilon > 0$, then
\[
\P\left( |\beta_n - \E \beta_n| \geq (\L n)^{\frac{3d}{4} + \epsilon}\mbox{\rm \ \io} \right) = 0.
\]
As a consequence,
\[
\frac{\beta_n}{\E \beta_n} \asto 1.
\]

{\rm (c)}~If $\epsilon > 0$, then
\[
\P\left( |r_n - \E r_n| \geq (\L n)^{\frac{3d}{4} + \epsilon}\mbox{\rm \ \io} \right) = 0.
\]
As a consequence, if $d \geq 5$ then
\[
\frac{r_n}{\E r_n} \asto 1.
\]
\end{theorem}

\begin{proof}
(a)~Since $\E R_n \sim (\L n)^d / d!$ by \refT{T:typical counts}(b), the second assertion is indeed an immediate consequence of the first.  To prove the first assertion, we establish
\begin{equation}
\label{right}
\P\left( R_n \geq \E R_n + (\L n)^{\frac{3d}{4} + \epsilon}\mbox{\ \io} \right) = 0
\end{equation}
and
\begin{equation}
\label{left}
\P\left( R_n \leq \E R_n - (\L n)^{\frac{3d}{4} + \epsilon}\mbox{\ \io} \right) = 0.
\end{equation}

To prove~\eqref{right} we exploit the nondecreasingness of the sample paths of the process~$R$.  If $(b_n)$ is ultimately monotone nondecreasing and $(n_j)$ is any strictly increasing sequence of positive integers, then
\begin{equation}
\label{right subsequence}
\{R_n \geq b_n\mbox{\rm\ \io$(n)$}\} \subseteq \{R_{n_{j + 1}} \geq b_{n_j}\mbox{\rm\ \io$(j)$}\}.
\end{equation}
Now choose $b_n \equiv \E R_n + (\L n)^{\frac{3d}{4} + \epsilon}$ (which is clearly nondecreasing) and
$n_j \equiv \lfloor e^{j^{2 / d}} \rfloor$.  Observe for large~$j$ that $\L n_j = j^{2 / d} + O(e^{- j^{2 / d}})$, and hence from \refT{T:typical counts}(b) that
\begin{align*}
\E R_{n_j}
&= (\L n_j)^d \sum_{k = 0}^d \frac{(-1)^k \Gamma^{(k)}(1)}{k! (d - k)!} (\L n_j)^{-k} + o(1) \\
&= j^2 \sum_{k = 0}^d \frac{(-1)^k \Gamma^{(k)}(1)}{k! (d - k)!} j^{-2k/d} + o(1) \sim \frac{j^2}{d!}, \\
\E R_{n_{j + 1}}
&= (j + 1)^2 \sum_{k = 0}^d \frac{(-1)^k \Gamma^{(k)}(1)}{k! (d - k)!} (j + 1)^{-2k/d} + o(1) \\
&= [1 + O(j^{-1})] j^2 \sum_{k = 0}^d \frac{(-1)^k \Gamma^{(k)}(1)}{k! (d - k)!} j^{-2k/d} + o(1) \\
&= \E R_{n_j} + O(j^{-1} \E R_{n_j}) + o(1) = \E R_{n_j} + O(j).
\end{align*}
Observe also that
\[
b_{n_j} - \E R_{n_j} = (\L n_j)^{\frac{3d}{4} + \epsilon} \sim j^{\frac{3}{2} + \frac{2}{d} \epsilon};
\]
As a consequence of these two observations,
\[
b_{n_j} - \E R_{n_{j + 1}} = (b_{n_j} - \E R_{n_j}) - (\E R_{n_{j + 1}} - \E R_{n_j})
\sim j^{\frac{3}{2} + \frac{2}{d} \epsilon} > 0.
\]
Further, from \refT{T:typical counts}(b) we have
\[
\Var R_{n_{j + 1}} \sim \gamma_{d + 1, 0} (\L n_{j + 1})^d = \Theta(j^2).
\]
Hence, by Chebyshev's inequality,
\[
P(R_{n_{j + 1}} \geq b_{n_j}) \leq (b_{n_j} - \E R_{n_{j + 1}})^{-2} \Var R_{n_{j + 1}}
= \Theta(j^{- (1 + \frac{4}{d} \epsilon)}),
\]
which is summable.  The first Borel--Cantelli lemma now implies that
\[
\P(R_{n_{j + 1}} \geq b_{n_j}\mbox{\ \io$(j)$}) = 0,
\]
and then~\eqref{right subsequence} yields the desired~\eqref{right}.

The proof of~\eqref{left} is similar and again uses the nondecreasingness of the sample paths of~$R$.  If
$(b_n)$ is ultimately monotone nondecreasing and $(n_j)$ is any strictly increasing sequence of positive integers, then
\begin{equation}
\label{left subsequence}
\{R_n \leq b_n\mbox{\rm\ \io$(n)$}\} \subseteq \{R_{n_j} \leq b_{n_{j + 1}}\mbox{\rm\ \io$(j)$}\}.
\end{equation}
Now choose $b_n \equiv \E R_n - (\L n)^{\frac{3d}{4} + \epsilon}$ and, again,
$n_j \equiv \lfloor e^{j^{2 / d}} \rfloor$.  The sequence $(b_n)$ is ultimately monotone nondecreasing because it is known (\eg,\ \cite{Bai(2005)}) that
\begin{equation}
\label{sets}
\E R_n - \E R_{n - 1} = \P(X^{(n)}\mbox{\ sets a record}) = n^{-1} \E r_n \sim n^{-1} \frac{(\L n)^{d - 1}}{(d - 1)!},
\end{equation}
while also
\begin{align*}
(\L n)^{\frac{3d}{4} + \epsilon} - [\L (n - 1)]^{\frac{3d}{4} + \epsilon}
&\sim ({\tfrac{3d}{4} + \epsilon}) n^{-1} (\L n)^{{\frac{3d}{4} - 1 + \epsilon}} \\
&= o(n^{-1} (\L n)^{d - 1}),
\end{align*}
provided $\epsilon < d / 4$ (which we may assume without loss of generality), whence
\[
b_n - b_{n - 1} \sim n^{-1} \frac{(\L n)^{d - 1}}{(d - 1)!} > 0.
\]
Proceeding as for~\eqref{right}, by Chebyshev's inequality we have
\[
P(R_{n_j} \leq b_{n_{j + 1}}) \leq (\E R_{n_j} - b_{n_{j + 1}})^{-2} \Var R_{n_j}
= \Theta(j^{- (1 + \frac{4}{d} \epsilon)}),
\]
which is summable.  The first Borel--Cantelli lemma now implies that
\[
\P(R_{n_j} \leq b_{n_{j + 1}}\mbox{\ \io$(j)$}) = 0,
\]
and then~\eqref{left subsequence} yields the desired~\eqref{left}.

(b)~For $d = 1$, part~(b) follows from part~(a) because $r_n = 1$ for $n \geq 1$, so we assume $d \geq 2$.  The sample paths of~$\beta$, like those of~$R$, are nondecreasing.  Thus, in precisely the same fashion that part~(a) is proved using the mean and variance results from \refT{T:typical counts}(b), so one can prove
part~(b) using the mean and variance results from \refT{T:typical counts}(c).  A key technical detail in establishing the analogue of~\eqref{left} for the process~$\beta$ is this analogue of~\eqref{sets} [which follows immediately from~\eqref{sets} by use of concomitants]:
\begin{align*}
\E \beta_n - \E \beta_{n - 1} &= (\E R_n - \E R_{n - 1}) - (\E r_n - \E r_{n - 1}) \\
&= (1 + o(1)) n^{-1} \frac{(\L n)^{d - 1}}{(d - 1)!} - (1 + o(1)) n^{-1} \frac{(\L n)^{d - 2}}{(d - 2)!} \\
&\sim n^{-1} \frac{(\L n)^{d - 1}}{(d - 1)!}.
\end{align*}

(c)~We obtain part(c) by subtraction from parts (a)--(b):
\begin{align*}
\lefteqn{\P\left( |r_n - \E r_n| \geq (\L n)^{\frac{3 d}{4} + \epsilon}\mbox{\ \io} \right)} \\
&= \P\left( |(R_n - \E R_n) - (\beta_n - \E \beta_n)| \geq (\L n)^{\frac{3 d}{4} + \epsilon}\mbox{\ \io} \right) \\
&\leq \P\left( |R_n - \E R_n| \geq \half (\L n)^{\frac{3 d}{4} + \epsilon}\mbox{\ \io} \right)
 + \P\left( |\beta_n - \E \beta_n| \geq \half (\L n)^{\frac{3 d}{4} + \epsilon}\mbox{\ \io} \right) \\
 &\leq \P\left( |R_n - \E R_n| \geq (\L n)^{\frac{3 d}{4} + \frac{\epsilon}{2}}\mbox{\ \io} \right)
 + \P\left( |\beta_n - \E \beta_n| \geq (\L n)^{\frac{3 d}{4} + \frac{\epsilon}{2}}\mbox{\ \io} \right) = 0.
\end{align*}
This gives the first assertion.  Since $\E r_n \sim (\L n)^{d - 1} / (d - 1)!$ by \refT{T:typical counts}(a), the second assertion is indeed an immediate consequence of the first provided $3 d / 4 < d - 1$, \ie,\ $d \geq 5$.
\end{proof}

\begin{remark}
\label{r as}
(a)~In the proof of \refT{T:as}(a) we utilized Chebyshev's inequality.  Use of normal tail proabilities would give a sharper result, except that the error estimate in the Berry--Esseen theorem of \refT{T:typical counts}(b) is insufficiently sharp for that.

(b)~For
$d = 2, 3, 4$ we conjecture on the basis of simulations discussed in \refE{E:r as records-time} that the second conclusion
\[
r_n / \E r_n \asto 1,
\]
\ie,
\begin{equation}
\label{r}
r_n / (\L n)^{d - 1} \asto 1 / (d - 1)!,
\end{equation}
of \refT{T:as}(c) remains true.
We
do at least know from the first assertion in \refT{T:as}(c) that for any $\epsilon > 0$ we have
\begin{equation}
\label{remaining little}
r_n = O((\L n)^{\frac{3 d}{4} + \epsilon})\mbox{\ \as}
\end{equation}
In dimension $d = 2$
we can come close to~\eqref{r}, or at least to showing that $r_n = \Theta(\L n)$ \as\ \ Indeed, we can combine the representation of the distribution of $r_n$ as a Poisson-binomial sum with a Chernoff bound and the first Borel--Cantelli lemma to show that $r_n = O(\L n)$ \as,\ and \refT{T:r} gives $r_n = \Omega((\L n) / (\L_2 n))$ \as
\ignore{
{\bf Here, for the record, is the outline of a straightforward proof of $r_n = O(\L n)\mbox{\ \as}$; I think it should be suppressed, rather than fleshed out, in our final draft.}  We show that, almost surely, $r_n \leq 3 \L n$ for all large~$n$ (and note in passing that the constant~$3$ here can be replaced by any constant larger than~$e$).

The number $r_n$ of current records has the same distribution as the total number $R_n$ of records set through time~$n$ in dimension 1.  As we know, $R_n$ is a Poisson-binomial sum.  It is standard and simple (and reproduced in Lemmas 2.6--2.7 of Jason Matterer's dissertation, sent from JAF to DQN by email on 2018.11.02; see the last sentence in the proof of his Lemma 2.7 and use $k = \lceil c H_n \rceil$ with $c > e$) to use Chernoff to bound the right tail of $R_n$.  Combined with the first Borel--Cantelli lemma, that gives the claimed result.
}
\end{remark}

\section{Time change}
\label{S:time}

It is natural to wonder about the appearance of the record-setting frontier (even in dimension~$2$) when
many observations, or (equivalently) many records, have been generated.  \refF{fig:simulation} displays the record-setting frontier for one trial after 10,000 bivariate records had been generated, at which point results such as those in \refS{S:intro} suggest themselves.  According to \refT{T:typical counts}(b) [or \refP{P:time}(a2)], had this been done naively, by generating observations $X^{(i)}$ and waiting for new records to be set, it would have taken roughly $10^{61}$ observations to obtain 10,000 records.  Instead, only the records were generated, using the importance-sampling scheme described and analyzed in~\cite{Fillgenerating(2018)}.

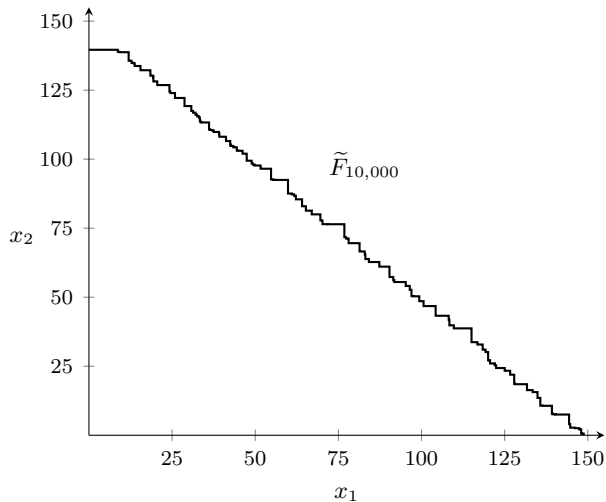
\begin{figure}[htb]
\begin{tikzpicture}[scale=1.]
  \begin{axis}[
    xmin=0, xmax=155,
    ymin=0, ymax=155,
    x label style={at={(axis description cs:0.5,-0.1)},anchor=north},
    y label style={at={(axis description cs:-.13,.5)},rotate=270,anchor=north},
    xtick = {25,50,75,100,125,150},
    ytick = {25,50,75,100,125,150},
    xticklabels={25,50,75, 100,125, 150},
    yticklabels = {25, 50,75,100,125, 150},
    ticklabel style = {font=\scriptsize},
    xlabel={\footnotesize $x_1$},
    ylabel={\footnotesize $x_2$},
    axis x line=bottom,
    axis y line=left
    ]
\addplot[mark=none,color=black,thick] coordinates {
(0, 139.619 )( 8.6835
, 139.619 )( 8.6835 , 138.953 ) ( 8.8741 , 138.953 )( 8.8741 ,
138.832 ) ( 9.16347 , 138.832 )( 9.16347 , 138.747 ) ( 11.9457 ,
138.747 )( 11.9457 , 135.637 ) ( 12.8197 , 135.637 )( 12.8197 ,
134.869 ) ( 13.5316 , 134.869 )( 13.5316 , 134.688 ) ( 13.7647 ,
134.688 )( 13.7647 , 133.746 ) ( 15.5036 , 133.746 )( 15.5036 ,
132.208 ) ( 18.4735 , 132.208 )( 18.4735 , 130.24 ) ( 19.344 ,
130.24 )( 19.344 , 129.637 ) ( 19.3749 , 129.637 )( 19.3749 ,
128.164 ) ( 20.5514 , 128.164 )( 20.5514 , 126.834 ) ( 24.2125 ,
126.834 )( 24.2125 , 124.481 ) ( 24.4874 , 124.481 )( 24.4874 ,
123.951 ) ( 25.8762 , 123.951 )( 25.8762 , 122.163 ) ( 28.7392 ,
122.163 )( 28.7392 , 119.236 ) ( 30.7809 , 119.236 )( 30.7809 ,
117.468 ) ( 31.3101 , 117.468 )( 31.3101 , 116.838 ) ( 32.0379 ,
116.838 )( 32.0379 , 116.111 ) ( 32.5296 , 116.111 )( 32.5296 ,
115.531 ) ( 33.0526 , 115.531 )( 33.0526 , 115.112 ) ( 33.3452 ,
115.112 )( 33.3452 , 113.926 ) ( 33.5192 , 113.926 )( 33.5192 ,
113.347 ) ( 34.0154 , 113.347 )( 34.0154 , 113.293 ) ( 36.1155 ,
113.293 )( 36.1155 , 110.84 ) ( 36.3134 , 110.84 )( 36.3134 ,
110.576 ) ( 37.1146 , 110.576 )( 37.1146 , 110.336 ) ( 37.5161 ,
110.336 )( 37.5161 , 109.856 ) ( 39.0965 , 109.856 )( 39.0965 ,
109.361 ) ( 39.1982 , 109.361 )( 39.1982 , 108.12 ) ( 41.2138 ,
108.12 )( 41.2138 , 106.558 ) ( 42.4826 , 106.558 )( 42.4826 ,
105.476 ) ( 42.524 , 105.476 )( 42.524 , 104.92 ) ( 42.7177 , 104.92
)( 42.7177 , 104.841 ) ( 42.739 , 104.841 )( 42.739 , 104.767 ) (
43.4227 , 104.767 )( 43.4227 , 104.31 ) ( 44.4536 , 104.31 )(
44.4536 , 103.252 ) ( 44.6386 , 103.252 )( 44.6386 , 103.037 ) (
46.1883 , 103.037 )( 46.1883 , 101.976 ) ( 47.443 , 101.976 )(
47.443 , 99.4383 ) ( 48.9235 , 99.4383 )( 48.9235 , 98.4571 ) (
49.3092 , 98.4571 )( 49.3092 , 98.0371 ) ( 49.774 , 98.0371 )(
49.774 , 97.8187 ) ( 49.8035 , 97.8187 )( 49.8035 , 97.7128 ) (
51.6117 , 97.7128 )( 51.6117 , 96.5417 ) ( 54.757 , 96.5417 )(
54.757 , 92.6536 ) ( 55.5015 , 92.6536 )( 55.5015 , 92.4824 ) (
59.857 , 92.4824 )( 59.857 , 87.5588 ) ( 60.9829 , 87.5588 )(
60.9829 , 87.2631 ) ( 61.6822 , 87.2631 )( 61.6822 , 86.7724 ) (
62.2196 , 86.7724 )( 62.2196 , 85.4493 ) ( 64.0775 , 85.4493 )(
64.0775 , 82.9518 ) ( 65.2361 , 82.9518 )( 65.2361 , 81.3453 ) (
66.9776 , 81.3453 )( 66.9776 , 81.2816 ) ( 66.9891 , 81.2816 )(
66.9891 , 80.0105 ) ( 69.5915 , 80.0105 )( 69.5915 , 77.7917 ) (
70.228 , 77.7917 )( 70.228 , 76.4585 ) ( 71.8602 , 76.4585 )(
71.8602 , 76.4328 ) ( 76.8009 , 76.4328 )( 76.8009 , 71.7273 ) (
77.3354 , 71.7273 )( 77.3354 , 71.2488 ) ( 78.0789 , 71.2488 )(
78.0789 , 69.5738 ) ( 81.3703 , 69.5738 )( 81.3703 , 66.5731 ) (
82.9366 , 66.5731 )( 82.9366 , 65.5116 ) ( 83.116 , 65.5116 )(
83.116 , 63.8384 ) ( 84.1717 , 63.8384 )( 84.1717 , 62.7488 ) (
87.3047 , 62.7488 )( 87.3047 , 61.0682 ) ( 90.3656 , 61.0682 )(
90.3656 , 57.3202 ) ( 91.4925 , 57.3202 )( 91.4925 , 56.1526 ) (
91.8287 , 56.1526 )( 91.8287 , 55.4975 ) ( 95.2544 , 55.4975 )(
95.2544 , 54.08 ) ( 96.4635 , 54.08 )( 96.4635 , 52.6625 ) ( 96.9768
, 52.6625 )( 96.9768 , 50.4186 ) ( 97.5024 , 50.4186 )( 97.5024 ,
50.3255 ) ( 99.3293 , 50.3255 )( 99.3293 , 48.5651 ) ( 100.615 ,
48.5651 )( 100.615 , 46.7881 ) ( 104.25 , 46.7881 )( 104.25 , 43.279
) ( 108.177 , 43.279 )( 108.177 , 41.8299 ) ( 108.375 , 41.8299 )(
108.375 , 39.8163 ) ( 109.725 , 39.8163 )( 109.725 , 38.7162 ) (
115.033 , 38.7162 )( 115.033 , 33.7155 ) ( 116.813 , 33.7155 )(
116.813 , 32.8684 ) ( 118.388 , 32.8684 )( 118.388 , 31.0066 ) (
119.271 , 31.0066 )( 119.271 , 30.1307 ) ( 120.027 , 30.1307 )(
120.027 , 27.1157 ) ( 120.522 , 27.1157 )( 120.522 , 27.0183 ) (
120.607 , 27.0183 )( 120.607 , 26.0109 ) ( 121.771 , 26.0109 )(
121.771 , 26.0047 ) ( 122.06 , 26.0047 )( 122.06 , 25.3338 ) (
122.433 , 25.3338 )( 122.433 , 24.3521 ) ( 125.002 , 24.3521 )(
125.002 , 23.3576 ) ( 126.647 , 23.3576 )( 126.647 , 21.9386 ) (
127.865 , 21.9386 )( 127.865 , 18.7059 ) ( 128.048 , 18.7059 )(
128.048 , 18.4906 ) ( 131.709 , 18.4906 )( 131.709 , 16.348 ) (
133.422 , 16.348 )( 133.422 , 15.6255 ) ( 134.877 , 15.6255 )(
134.877 , 13.5231 ) ( 135.715 , 13.5231 )( 135.715 , 10.7266 ) (
136.491 , 10.7266 )( 136.491 , 10.6685 ) ( 139.202 , 10.6685 )(
139.202 , 7.92818 ) ( 139.411 , 7.92818 )( 139.411 , 7.59583 ) (
140.285 , 7.59583 )( 140.285 , 7.50318 ) ( 144.367 , 7.50318 )(
144.367 , 4.05198 ) ( 144.705 , 4.05198 )( 144.705 , 3.43208 ) (
144.705 , 3.43208 )( 144.705 , 2.73945 ) ( 146.169 , 2.73945 )(
146.169 , 2.55625 ) ( 147.446 , 2.55625 )( 147.446 , 2.11898 ) (
148.013 , 2.11898 )( 148.013 , 0.755152 ) ( 148.059 , 0.755152 )(
148.059 , 0.598016 ) ( 148.754 , 0.598016 )( 148.754 , 0 )
};
\draw (83,97.5) node[color=black] {\footnotesize $\widetilde{F}_{10,000}$};
\end{axis}
\end{tikzpicture}

\caption{
Record frontier $\widetilde{F}_{10,000}$ after $10,000$ records generated using the importance-sampling algorithm
described in~\cite{Fillgenerating(2018)}.
\label{fig:simulation}
}

\end{figure}

The record-setting region process $(\mbox{RS}_n)$, and therefore also the frontier process $(F_n)$ we
have studied in earlier sections, is adapted to the natural filtration for the process $C = (C_n)_{n \geq 0}$, where $C_n = (C^{(1)}_n, \dots, C^{(r_n)}_n)$ is the $r_n$-tuple of
remaining records at time~$n$ in order of creation.  Let $T_0 = 0$, and for $m \geq 1$ let $T_m$ denote the $m$th record-creation epoch; note that~$C$ remains constant over each of the time-intervals $[T_{m - 1}, T_m)$, $m \geq 1$.  Fill and Naiman~\cite{Fillgenerating(2018)} don't simulate the \iid\ observations process $X^{(1)}, X^{(2)}, \dots$ (that is, they don't work in ``observations-time''), but rather simulate the process $\tC = (\tC_m)_{m \geq 0}$, where $\tC_m := C_{T_m}$ [and hence the processes $(\widetilde{{\mbox RS}}_m := \mbox{RS}_{T_m})$ and
 $(\tB_m := F_{T_m})$] (that is, they work in ``records-time'').  The following goal thus naturally arises: Translate results about~$C$ to results about~$\tC$.

The keys to doing so are (i)~monotonicity of the sample paths of various processes of interest (such as $F^+$ and $F^-$) and (ii)~the switching relation
\begin{equation}
\label{switch}
\{T_m \leq n\} = \{R_n \geq m\}.
\end{equation}
The switching relation enables us to obtain information about the record-creation times $T_m$
from the records-counts Theorems \ref{T:typical counts}(b) and \ref{T:as}(a).  The following proposition is not the most elaborate result which can be obtained in such fashion, but it will suffice for our purposes.

\begin{proposition}
\label{P:time}
Let $T_m$ denote the $m^{\mbox{\rm \scriptsize th}}$ epoch at which a record is set, and let~$\gamma$ denote the Euler--Mascheroni constant.
\smallskip

\noindent
{\rm (a)~\underline{Typical behavior as $m \to \infty$}:\ }
\smallskip

{\rm (a1)}~If $d = 1$, then
\[
\frac{\L T_m - (m - \gamma)}{m^{1/2}} \Lcto\mbox{\rm \ standard normal}.
\]

{\rm (a2)}~If $d = 2$, then
\[
\frac{\L T_m - [(2 m)^{1 / 2} - \gamma]}{\left( \frac{\pi^2}{6} + \frac{1}{2} \right)^{1/2}} \Lcto\mbox{\rm \ standard normal}.
\]

{\rm (a3)}~If $d \geq 3$, then
\[
\L T_m - [(d! m)^{1 / d} - \gamma] \Pto 0.
\]
\smallskip

\noindent
{\rm (b)~\underline{Almost sure behavior as $m \to \infty$}:\ }
\smallskip

{\rm (b1)}~For every $d \geq 1$ we have
\[
\frac{\L T_m}{(d! m)^{1/d}} \asto 1.
\]

{\rm (b2)}~If $d \geq 5$, then
\[
\L T_m - [(d! m)^{1/d} - \gamma] \asto 0.
\]
\end{proposition}

Concerning elaborations on \refP{P:time}(b2), see \refR{R:last}(b).

\begin{proof}
Fix $d \geq 1$.
\smallskip

(a)~Given $\epsilon > 0$, by the switching relation~\eqref{switch} and \refT{T:typical counts}(b) we
have
\begin{align}
\label{smallprob}
\P(\L T_m - [(d! m)^{1 / d} - \gamma] > \epsilon)
&= \P(T_m > \exp[(d! m)^{1 / d} - \gamma + \epsilon]) \\
&= \P(T_m > n) = \P(R_n < m) \nonumber \\
&= \Phi\left( \frac{m - \E R_n}{\sqrt{\Var R_n}} \right) + o(1) \nonumber
\end{align}
as $m \to \infty$, where
$0 \leq \epsilon_m = o(1)$ is chosen as small as possible to make
$n \equiv n_m := \exp[(d! m)^{1 / d} - \gamma + \epsilon - \epsilon_m]$ an integer.  But
$\L n = (d! m)^{1 / d} - \gamma + \epsilon - o(1)$, so
\[
(\L n)^d = d! m [1 - (1 + o(1)) (\gamma - \epsilon) d (d! m)^{-1/d}] \quad \mbox{and} \quad
(\L n)^{-1} \sim (d! m)^{-1/d},
\]
and hence by \refT{T:typical counts}(b)
\begin{align*}
\E R_n
&= \frac{(\L n)^d}{d!} [1 + (1 + o(1)) \gamma d (\L n)^{-1}] \\
&= m [1 - (1 + o(1)) (\gamma - \epsilon) d (d! m)^{-1/d}] [1 + (1 + o(1)) \gamma d (d! m)^{-1/d}] \\
&= m [1 + (1 + o(1)) \epsilon d (d! m)^{-1/d}] = m + (1 + o(1)) \epsilon d (d!)^{-1/d} m^{(d - 1) / d}
\end{align*}
and
\[
\sqrt{\Var R_n} \sim \sqrt{\gamma_{d + 1, 0}} (\L n)^{d / 2} \sim (\gamma_{d + 1, 0}\,d!\,m)^{1/2}
=  \Theta(m^{1/2}).
\]
Thus $(m - \E R_n) / \sqrt{\Var R_n}$ is negative and of magnitude
$\Theta(m^{\frac{d-1}{d} - \frac{1}{2}})$.
\smallskip

(a3)~If $d \geq 3$, it follows that the probability~\eqref{smallprob} tends to~$0$, and similarly
\[
\P(\L T_m - [(d! m)^{1 / d} - \gamma] \leq - \epsilon) \to 0,
\]
yielding the claimed convergence in probability.
\smallskip

(a2)~If $d = 2$, then the same calculations show that for any real~$x$ we have
\[
\P(\L T_m - [(2 m)^{1 / 2} - \gamma] > x)
= \Phi\left( - \gamma_{3, 0}^{-1/2} x \right) + o(1),
\]
yielding the claimed CLT, since from~\cite{Bai(2005)}, $\gamma_{3, 0} = \frac{\pi^2}{6} + \frac{1}{2}$.
\smallskip

(a1)~If $d = 1$, then the same calculations show that for any real~$x$ we have
\[
\P(\L T_m - [m - \gamma] > x) = \Phi\left( - (1 + o(1)) \frac{x}{(\gamma_{2, 0}\,m)^{1/2}} \right) + o(1),
\]
yielding the claimed CLT, since $\gamma_{2, 0} = 1$.
\smallskip

(b1)~This follows readily from the conclusion $R_n / \E R_n \asto 1$ of \refT{T:as}(a) by first recalling from \refT{T:typical counts}(b) that $\E R_n \sim (\L n)^d / d!$; then setting $n = T_m$, noting $R_{T_m} = m$; and finally taking $-d^{-1}$ powers.
\smallskip

(b2)~According to \refT{T:as}, if $\epsilon > 0$ then as $n \to \infty$ we \as\ have
\[
R_n = \rho_n + O((\L n)^{\frac{3 d}{4} + \epsilon}),
\]
where~$\rho$ is the mean function for~$R$.  In particular, setting $n = T_m$, as $m \to \infty$ we \as\ have
\[
m = \rho_{T_m} + O((\L T_m)^{\frac{3 d}{4} + \epsilon}).
\]
If $d \geq 5$, then $d - 1 > (3 d) / 4$ and thus [from \refT{T:typical counts}(b)] almost surely
\[
m = \frac{(\L T_m)^d}{d!} [1 + (1 + o(1)) \gamma d (\L T_m)^{-1}],
\]
which implies
\[
(d! m)^{1/d} = (\L T_m) [1 + (1 + o(1)) \gamma (\L T_m)^{-1}] = \L T_m + \gamma + o(1),
\]
as desired.
\end{proof}

\begin{example}
\label{E:r as records-time}
Here
is a first illustration of the usefulness of \refP{P:time} in connection with the simulations of records discussed at the outset of this section.  Define $\tr_m := r_{T_m}$.  From these simulations it is reasonable to conjecture that
\begin{equation}
\label{conjecture}
\frac{\tr_m}{(d! m)^{(d - 1) / d}} \asto \frac{1}{(d - 1)!}\mbox{\ as $m \to \infty$}.
\end{equation}
But we now show that the records-time conjecture~\eqref{conjecture} is in fact equivalent to the
observations-time conjecture~\eqref{r}---and therefore both conjectures are [by \refT{T:as}(c) and the expected value asymptotics in \refT{T:typical counts}(a)] true at least for $d \geq 5$.

Indeed, \eqref{conjecture} follows immediately from~\eqref{r} by substitution of $T_m$ for~$n$ and use of \refP{P:time}(b1).  To sketch a proof of the converse, consider the ratio on the left in~\eqref{r} for $T_m \leq n < T_{m + 1}$.  For the numerator of the ratio, note that $r_n = r_{T_m}$.  Use $T_m \leq n < T_{m + 1}$ in the denominator to get upper and lower bounds on the ratio, and then use \refP{P:time}(b1) to relate the upper and lower bounds on the ratio in~\eqref{r} to the ratio in~\eqref{conjecture}.
 \end{example}
 \medskip

We can now translate results of \refS{S:intro} from observations-time to records-time (the main goal of this section being to translate \refT{T:W} about frontier width in this fashion), but [because of the limitation of \refP{P:time}(b2)] we only know how to translate some of our almost sure results when $d \geq 5$.

\begin{theorem}
\label{T:t+}
Consider the process $\tB^+$ defined by $\tB^+_m := F^+_{T_m}$.
\medskip

\noindent
{\rm (a)~\underline{Typical behavior of $\tB^+$}:\ }
\smallskip

{\rm (a1)}~For any $d \geq 2$ we have
\[
\frac{\tB_m^+ - (d! m)^{1/d}}{\L m} \Pto 1 - d^{-1}.
\]

{\rm (a2)}~If $d \geq 3$ we have the following convergence in law to Gumbel:
\[
\tB_m^+ - [(d! m)^{1/d} + (1 - d^{-1}) \L m + \L d - d^{-1} \L(d!) - \gamma] \Lcto G.
\]
\medskip

\noindent
{\rm (b)~\underline{Almost sure behavior for $\tB^+$}:\ }
\smallskip

{\rm (b1)}~For any $d \geq 1$ we have
\[
\tB^+_m \sim (d! m)^{1 / d} \mbox{\rm \ \as}
\]

{\rm (b2)}~If $d \geq 5$, then
\[
\liminf \frac{\tB_m^+ - (d! m)^{1/d}}{\L m} = 1 - d^{-1} < 1 =
\limsup \frac{\tB_m^+ - (d! m)^{1/d}}{\L m} \mbox{\rm \ \as}
\]
\end{theorem}
\smallskip

\begin{proof}
(a2)~Assume that $d \geq 3$ and let
\[
\tG_m := \tB_m^+ - [(d! m)^{1/d} + (1 - d^{-1}) \L m + \L d - d^{-1} \L(d!) - \gamma].
\]
Given $x \in \bbR$ and $\epsilon > 0$, we will show that
\begin{equation}
\label{G}
\P(\tG_m \leq x) \geq \P(G \leq x - \epsilon) - o(1),
\end{equation}
and a similar proof establishes $\P(\tG_m \leq x) \leq \P(G \leq x + \epsilon) + o(1)$.  Letting $m \to \infty$ and then $\epsilon \downarrow 0$ completes the proof of~(a2), and~(a1) is a simple consequence.  

We now prove~\eqref{G}.  By \refP{P:time}(a3) and nondecreasingness of the sample paths of $F^+$, we have
\[
\P(\tG_m \leq x)
\geq \P\Big( F^+_n \leq x + (d! m)^{1/d} + (1 - d^{-1}) \L m + \L d - d^{-1} \L(d!) - \gamma \Big) - o(1),
\]
where $n \equiv n_m = \lfloor \exp[(d! m)^{1/d} - \gamma + \epsilon] \rfloor$.  Observe that
\[
\L n = (d! m)^{1/d} - \gamma + \epsilon - o(1) \quad \mbox{and} \quad
\L_2 n = d^{-1} [\L m + \L(d!)] + o(1),
\]
and so
\begin{align*}
\lefteqn{\hspace{-.3in}\L n + (d - 1) \L_2 n - L((d -1)!)} \\
&= (d! m)^{1/d} + (1 - d^{-1}) \L m + \L d - d^{-1} \L(d!) - \gamma + \epsilon - o(1).
\end{align*}
Thus, making use of Theorem \ref{T:+}(a), we arrive at
\begin{align*}
\lefteqn{\hspace{-.1in}\P(\tG_m \leq x)} \\
&\geq \P\!\Big( F^+_n - [ \L n + (d - 1) \L_2 n - L((d -1)!] \leq x - \epsilon + o(1) \Big) - o(1) \\
&= \P(G \leq x - \epsilon) - o(1),
\end{align*}
as desired.

(a1)~We have already proved (a1) for $d \geq 3$.  A similar proof establishes (a1) if $d = 2$.

(b1)~By \refC{C:+}(b) and \refP{P:time}(b1), the following asymptotic equivalences hold \as:
\[
\tB^+_m = F^+_{T_m} \sim \L T_m \sim (d! m)^{1 / d}. 
\]

(b2)~One checks easily for $b \geq 0$ that $(b - \L n) / \L_2 n$ decreases for $n \geq 15$, and so
$(F^+_n - \L n) / \L_2 n$ decreases over each of the time-intervals $[T_{m - 1}, T_m)$ with~$m$ large.  (It is sufficient to choose $m \geq 16$.)  It follows that
\begin{equation}
\label{limsup}
\limsup_{n \to \infty} \frac{F^+_n - \L n}{\L_2 n} = \limsup_{m \to \infty} \frac{\tB^+_m - \L T_m}{\L_2 T_m}
\end{equation}
and
\begin{align*}
\liminf_{n \to \infty} \frac{F^+_n - \L n}{\L_2 n}
&= \liminf_{m \to \infty} \frac{F^+_{T_m - 1} - \L(T_m - 1)}{\L_2(T_m - 1)} \\
&= \liminf_{m \to \infty} \frac{\tB^+_{m - 1} - \L(T_m - 1)}{\L_2(T_m - 1)} \\
&= \liminf_{m \to \infty} \frac{\tB^+_m - \L T_{m + 1} + o(1)}{\L_2 T_{m + 1} - o(1)}.
\end{align*}
But, by \refP{P:time}(b2), almost surely
\[
\L T_{m + 1} = [d! (m + 1)]^{1/d} - \gamma + o(1) = d! m^{1/d} + O(1)
\]
and hence
\[
\L_2 T_{m + 1} = d^{-1} \L m + O(1),
\]
whence
\begin{align*}
\liminf_{n \to \infty} \frac{F^+_n - \L n}{\L_2 n}
&= \liminf_{m \to \infty} \frac{\tB^+_m - \L T_{m + 1} + o(1)}{\L_2 T_{m + 1} - o(1)} \\
&= \liminf_{m \to \infty} \frac{\tB^+_m - d! m^{1/d} + O(1)}{d^{-1} \L m + O(1)} \\
&= d \liminf_{m \to \infty} \frac{\tB^+_m - d! m^{1/d}}{\L m};
\end{align*}
similarly, by \eqref{limsup},
\[
\limsup_{n \to \infty} \frac{F^+_n - \L n}{\L_2 n}
= d \limsup_{m \to \infty} \frac{\tB^+_m - d! m^{1/d}}{\L m}.
\]
The desired result now follows from \refC{C:+}(b).
\end{proof}
\smallskip

\begin{remark}
\label{R:interval t+}
In the same manner as \refR{R:interval+}, one can show that the set of limit points of the sequence
$[\tB_m^+ - (d! m)^{1/d}] / \L m$ is for $d \geq 5$ almost surely the closed interval
$[1 - d^{-1}, 1]$.
\end{remark}

\begin{theorem}
\label{T:t-}
Consider the process $\tB^-$ defined by $\tB^-_m := F^-_{T_m}$.
\medskip

{\rm (a)~\underline{Typical behavior of $\tB^-$}:\ }If $d \geq 2$, then
\[
\P(\tB_m^- \leq (d! m)^{1/d} - 3 \L_2 m) \to 0
\]
and
\[
\P(\tB_m^- \geq (d! m)^{1/d} + c_m) \to 0\mbox{\rm \ if $c_m \to \infty$}.
\]
As a consequence,
\[
\frac{\tB_m^- - (d! m)^{1/d}}{\L m} \Pto 0.
\]
\smallskip

{\rm (b)~\underline{Almost sure behavior for $\tB^-$}:\ }If $d \geq 5$, then
\[
\lim \frac{\tB_m^- - (d! m)^{1/d}}{\L m} = 0\mbox{\rm \ \as}
\]
\end{theorem}
\smallskip

\begin{proof}
(a)~Recalling \refR{R:3} to provide some flexibility, part~(a) follows from \refT{T:-}(a) in much the same way that \refT{T:t+}(a) followed from \refT{T:+}(a) [and \refC{C:+}(a)].  In the interest of brevity, we omit the routine details.

(b)~In the same way that \refT{T:t+}(b) followed from \refC{C:+}(b), so part~(b) follows from \refC{C:-}(b).
\end{proof}

We come finally to our main focus of this section, the process~$\tW$.
\medskip

\begin{theorem}
\label{T:tW}
Consider the process $\tW$ defined by $\tW_m := W_{T_m}$.
\medskip

{\rm (a)~\underline{Typical behavior of $\tW$}:\ }For every $d \geq 1$ we have
\[
\frac{\tW_m}{\L m} \Pto 1 - d^{-1}.
\]

{\rm (b)~\underline{Almost sure behavior for~$\tW$}:\ }If $d \geq 2$, then
\[
\liminf \frac{\tW_m}{\L m} = 1 - d^{-1} < 1 =
\limsup \frac{\tW_m}{\L m} \mbox{\rm \ \as}
\]
and, in particular,
\[
\tW_m = \Theta(\L m) \mbox{\rm \ \as}
\]
\end{theorem}

\begin{proof}
Part~(a), and part~(b) for $d \geq 5$, follow immediately by subtraction from the two preceding theorems about 
$\tB^+$ and $\tB^-$ [and by the triviality of part~(a) for $d = 1$].  We next present an argument that establishes part~(b) for \emph{all} $d \geq 2$.

In the proofs of Theorems~\ref{T:t+}(b) and \ref{T:t-}(b), the only use of the assumption $d \geq 5$ is in the application of \refP{P:time}(b2).  From the computations prior to the application together with application of 
\refP{P:time}(b1) for the denominators, we almost surely have
\begin{align}
\label{tB+}
\limsup_{m \to \infty} \frac{\tB^+_m - \L T_m}{\L m} &= 1, \qquad
\liminf_{m \to \infty} \frac{\tB^+_m - \L T_{m + 1}}{\L m} = 1 - d^{-1}, \\
\nonumber
\limsup_{m \to \infty} \frac{\tB^-_m - \L T_m}{\L m} &= 0, \qquad
\liminf_{m \to \infty} \frac{\tB^-_m - \L T_{m + 1}}{\L m} = 0.
\end{align}
From the two results here about $\tB^-$, it follows quickly using the monotonicity of the paths of $F^-$ that \as
\begin{equation}
\label{tB-}
\lim_{m \to \infty} \frac{\tB^-_m - \L T_m}{\L m} = 0, \qquad  
\lim_{m \to \infty} \frac{\tB^-_m - \L T_{m + 1}}{\L m} = 0.
\end{equation}
Now subtract the equations in~\eqref{tB-} from the corresponding equations in~\eqref{tB+} to complete the proof of part~(b).
 \end{proof}

\begin{remark}
\label{R:last}
(a)~Using \refR{R:interval t+}, for $d \geq 5$ \refT{T:tW}(b) can 
be strengthened to the conclusion that the set of limit points of the sequence $\tW_m / \L m$ is almost surely the closed interval $[1 - d^{-1}, 1]$.  We have not investigated whether this result can be extended to $d = 2, 3, 4$.

(b)~Equation~\eqref{tB-} has 
the independently interesting corollary that
\begin{equation}
\label{LTm difference}
\L T_{m + 1} - \L T_m = o(\L m)\mbox{\ \as}
\end{equation}
for $d \geq 2$.  For $d = 1$,  it follows from the last sentence in \cite[Sec.~2.5]{Arnold(1998)} that
\[
\L T_{m + 1} - \L T_m = O((m \L_2 m)^{1/2}).
\]

For $d \geq 5$ we know the stronger [than~\eqref{LTm difference}] result
\[
L T_{m + 1} - \L T_m  = o(1) \mbox{\ \as}
\]
from \refP{P:time}(b2).  Even stronger results are available for larger values of~$d$.  For example, if 
$d \geq 9$ (so that $d - 2 > \frac{3}{4} d$), then the proof of \refP{P:time}(b2) can be extended to yield
\begin{equation}
\label{stronger LTm}
\L T_m = (d! m)^{1/d} - \gamma + (1 + o(1)) c_d m^{-1/d} \mbox{\rm \ \as}
\end{equation}
for a constant $c_d$ that can be computed explicitly.  Then~\eqref{stronger LTm} implies
\[
\L T_{m + 1} - \L T_m  = O(m^{-1/d})\mbox{\rm \ \as}
\]
\end{remark}

\begin{acks}
We thank Vince Lyzinski and Fred Torcaso for helpful comments.
\end{acks}

\bibliography{records}
\bibliographystyle{plain}

\end{document}